\documentclass[a4paper,11pt,reqno]{amsart}
\usepackage{graphicx}
\usepackage[latin1]{inputenc}
\usepackage{mathrsfs}
\usepackage{dsfont}
\usepackage{hyperref}
\usepackage{amsmath}
\usepackage{amssymb}
\usepackage{amsthm}
\usepackage{amsfonts}
\usepackage{amstext}
\usepackage{amsopn}
\usepackage{amsxtra}
\usepackage{mathrsfs}
\usepackage{dsfont}
\usepackage{esint}
\usepackage{enumitem}

\newtheorem{thm}{Theorem}
\newtheorem{conjecture}{Conjecture}
\newtheorem{lemma}[thm]{Lemma}

\newcommand{\R}{\mathbb{R}}
\newcommand{\Z}{\mathbb{Z}}
\newcommand{\C}{\mathbb{C}}

\newcommand{\ii}{\infty}
\newcommand\1{{\ensuremath {\mathds 1} }}
\renewcommand\phi{\varphi}

\newcommand{\wto}{\rightharpoonup}

\newcommand{\cV}{\mathcal{V}}

\newcommand{\cD}{\mathcal{D}}
\newcommand{\cW}{\mathcal{W}}

\newcommand\pscal[1]{{\ensuremath{\left\langle #1 \right\rangle}}}
\newcommand{\norm}[1]{ \left\| #1 \right\|}

\renewcommand{\geq}{\geqslant}
\renewcommand{\leq}{\leqslant}

\renewcommand{\tilde}{\widetilde}

\newcommand{\be}{\begin{equation}}
\newcommand{\ee}{\end{equation}}
\newcommand{\bq}{\begin{equation}}
\newcommand{\eq}{\end{equation}}

\newcommand{\un}{{1\!\!1}_4}

\newcommand{\Frac}{\displaystyle \frac}

\newcommand{\Sup}{\displaystyle \sup}

\newcommand{\eps}{\varepsilon}
\usepackage{color}

\newcommand{\nn}{\nonumber}

\newcommand{\Spec}{{\rm Sp}}

\title[Dirac-Coulomb operators with general charge distribution I]{Dirac-Coulomb operators with general charge distribution\\[0.2cm] I. Distinguished extension and min-max formulas}

\author[M.J. Esteban]{Maria J. Esteban}
\address{CEREMADE, CNRS, Université Paris-Dauphine, PSL Research University, Place de Lattre de Tassigny, 75016 Paris, France} 
\email{esteban@ceremade.dauphine.fr}

\author[M. Lewin]{Mathieu Lewin}
\address{CEREMADE, CNRS, Université Paris-Dauphine, PSL Research University, Place de Lattre de Tassigny, 75016 Paris, France} 
\email{mathieu.lewin@math.cnrs.fr}

\author[\'E. Séré]{\'Eric Séré}
\address{CEREMADE, Université Paris-Dauphine, PSL Research University, CNRS, Place de Lattre de Tassigny, 75016 Paris, France} 
\email{sere@ceremade.dauphine.fr}

\date{\today. Final version to appear in \emph{Annales Henri Lebesgue}.}

\DeclareRobustCommand{\SkipTocEntry}[5]{}

\begin{document}

\begin{abstract}
This paper is the first of a series where we study the spectral properties of Dirac operators with the Coulomb potential generated by any finite signed charge distribution $\mu$. We show here that the operator has a unique distinguished self-adjoint extension under the sole condition that $\mu$ has no atom of weight larger than or equal to one. Then we discuss the case of a positive measure and characterize the domain using a quadratic form associated with the upper spinor, following earlier works~\cite{EstLos-07,EstLos-08} by Esteban and Loss. This allows us to provide min-max formulas for the eigenvalues in the gap. In the event that some eigenvalues have dived into the negative continuum, the min-max formulas remain valid for the remaining ones. At the end of the paper we also discuss the case of multi-center Dirac-Coulomb operators corresponding to $\mu$ being a finite sum of deltas. 
\end{abstract}

\maketitle

\tableofcontents

\section{Introduction}

Relativistic effects play an important role in the description of quantum electrons in molecules containing heavy nuclei, even for not so large values of the nuclear charge. Without relativity, gold would have the same color as silver~\cite{GlaAmb-10}, mercury would be solid at room temperature~\cite{CalPahWorSch-13} and batteries would not work~\cite{ZalPyy-11}. This is due to the very strong Coulomb forces experienced by the core electrons, which can then attain large velocities of the order of the speed of light. 

A proper description of such atoms and molecules is based on the Dirac operator~\cite{Thaller,EstLewSer-08}. This is an order-one differential operator which has several famous mathematical difficulties, all associated with important physical features. For instance the spectrum of the free Dirac operator is not semi-bounded which prevents from giving an unambiguous definition of a  ``ground state'' and turns out to be related to the existence of the positron~\cite{EstLewSer-08}. In addition, the Dirac operator has a critical behavior with respect to the Coulomb potential $1/|x|$ which gives a bound $Z\leq 137$ on the highest possible charge of atoms in the periodic table, for point nuclei.

This paper is the first in a series where we study the spectral properties of Dirac operators with the Coulomb potential generated by \emph{any} finite (signed) measure $\mu$ representing an external charge:
\begin{equation}
\boxed{D_0-\mu\ast\frac1{|x|}=-i\sum_{j=1}^3\alpha_j\partial_{x_j}+\beta-\mu\ast\frac1{|x|}.}
\label{eq:general_form}
\end{equation}
One typical example is when the measure $\mu$ describes the $M$ nuclei in a molecule and this corresponds to
$$\mu=\sum_{m=1}^M\alpha Z_m\;\delta_{R_m}$$
where $R_m\in\R^3$ and $Z_m\in(0,\ii)$ are, respectively, the positions and charges of the $M$ nuclei, and where $\alpha\simeq 1/137$ is the Sommerfeld fine structure constant. For instance for water (H$_2$O) we have $M=3$, $Z_1=Z_2=1$ and $Z_3=8$. 
In practice, one should also take into account the Coulomb repulsion between the electrons. In mean-field type models such as Dirac-Fock or Kohn-Sham~\cite{EstSer-99}, this is described by a nonlinear potential  which is often more regular than the nuclear attraction. This leads us to consider the following class of (signed) measures
$$\mu=\sum_{m=1}^M\alpha Z_m\;\delta_{R_m}+\widetilde\mu,$$
where the measure $\widetilde\mu$ is more regular (for instance absolutely continuous with respect to the Lebesgue measure). 

In this paper, we first quickly recall existing results and then prove the existence of a distinguished self-adjoint extension for operators of the form~\eqref{eq:general_form}, under the sole assumptions that 
\begin{equation}
 |\mu|(\R^3)<\ii \qquad \text{and}\qquad |\mu(\{R\})|<1\quad \text{for all $R\in\R^3$.}
 \label{eq:hyp_intro}
\end{equation}
In particular we allow an infinite number $M=+\ii$ of atoms but assume that the total nuclear charge is bounded. We follow well established methods for singular Dirac operators~\cite{Schmincke-72b,Wust-73,Wust-75,Wust-77,Nenciu-76,Nenciu-77,KlaWus-78,Klaus-80b,Kato-83} but face several difficulties due to the generality of our measure $\mu$. In a second step we consider the particular case of a positive measure (or more generally a measure so that the Coulomb potential $\mu\ast|x|^{-1}$ is bounded from below) and we characterize the domain using a method introduced in~\cite{EstLos-07, EstLos-08} and recently generalized in~\cite{SchSolTok-20}. This method allows us to provide min-max formulas  for the eigenvalues in the gap $(-1,1)$, following~\cite{GriSie-99, DolEstSer-00, DolEstSer-00a, DolEstSer-03, DolEstSer-06,MorMul-15,Muller-16,EstLewSer-19,SchSolTok-20}. In the event that some eigenvalues have dived into the negative continuum, we prove that the min-max formulas remain valid for the remaining ones.

In a second article~\cite{EstLewSer-21b} we consider the problem of minimizing the first eigenvalue with respect to the (non-negative) charge distribution $\mu$ at fixed maximal charge $\nu$:
$$\lambda_1(\nu):=\inf_{\substack{\mu\geq0\\ \mu(\R^3)\leq\nu}}\lambda_1\left(D_0-\mu\ast\frac1{|x|}\right),$$
that is, we ask what is the lowest possible eigenvalue of all possible charge distributions with $\mu(\R^3)\leq\nu$. This problem is indeed our main motivation for studying Dirac operators of the type~\eqref{eq:general_form} with general measures $\mu$. We prove in~\cite{EstLewSer-21b} that there exists a critical coupling constant\footnote{In~\cite{EstLewSer-21b} $\overline\nu_1$ is simply denoted $\nu_1$.} 
$$\frac{2}{\frac\pi2+\frac2\pi}<\overline{\nu}_1\leq1$$ 
such that $\lambda_1(\nu)>-1$ for all $0\leq \nu<\overline\nu_1$, that is, the first eigenvalue cannot attain the bottom of the spectral gap if $\overline\nu_1-\mu(\R^3)$ remains positive. In addition, for $0\leq\nu<\overline\nu_1$ we prove the existence of a minimizing measure for $\lambda_1(\nu)$ and show that it concentrates on a set of Lebesgue measure zero. That the optimal measure is necessarily singular is the main justification for considering general charge distributions. 

It is well known that the first eigenvalue of the non-relativistic Schrödinger operator is concave in $\mu$, which implies that 
$$\inf_{\substack{\mu\geq0\\ \mu(\R^3)\leq\nu}}\lambda_1\left(-\frac\Delta2-\mu\ast\frac1{|x|}\right)=\lambda_1\left(-\frac\Delta2-\frac\nu{|x|}\right)=-\frac{\nu^2}{2}.$$
We conjecture that a similar equality holds in the Dirac case, which would imply $\overline\nu_1=1$ and $\lambda_1(\nu)=\sqrt{1-\nu^2}$. We mention below some physical implications that the validity of this conjecture would have for the electronic contribution to the potential energy surface of diatomic systems and other molecules.

\smallskip

The paper is organized as follows. In the next section we show that the operator~\eqref{eq:general_form} has a unique distinguished self-adjoint extension under the assumption~\eqref{eq:hyp_intro}, whereas in Section~\ref{sec:min-max} we discuss the domain and min-max formulas for the eigenvalues under the additional condition that $\mu\geq0$. The rest of the paper is then devoted to the proofs of our main results.

\bigskip

\noindent{\textbf{Acknowledgement.}} This project has received funding from the European Research Council (ERC) under the European Union's Horizon 2020 research and innovation programme (grant agreement MDFT No 725528 of M.L.), and from the Agence Nationale de la Recherche (grant agreement molQED).

\section{Distinguished self-adjoint extension for a general charge}\label{sec:self-adjointness}

In this section we give a meaning to the operator $D_0-\mu\ast|x|^{-1}$ for the largest possible class of bounded measures $\mu$. But first we need to clarify some notation. 

\subsection{Notation}
We work in a system of units for which $m=c=\hbar=1$. The free Dirac operator $D_0$ is given by
\begin{equation}
D_0\ = -i\boldsymbol{\alpha}\cdot\boldsymbol{\nabla} + \beta = \ - i\sum^3_{k=1} \alpha_k \partial_{x_k} \ + \ {\bf \beta},
\label{eq:def_Dirac}
\end{equation}
where $\alpha_1$, $\alpha_2$, $\alpha_3$ and $\beta$ are Hermitian matrices which  satisfy the following anticommutation relations:
\begin{equation*} \label{CAR}
\left\lbrace
\begin{array}{rcl}
 {\alpha}_k
{\alpha}_\ell + {\alpha}_\ell
{\alpha}_k  & = &  2\,\delta_{k\ell}\,\un,\\
 {\alpha}_k {\beta} + {\beta} {\alpha}_k
& = & 0,\\
\beta^2 & = & \un.
\end{array} \right. 
\end{equation*}
The usual representation in $2\times 2$ blocks is given by 
$$ \beta=\left( \begin{matrix} I_2 & 0 \\ 0 & -I_2 \\ \end{matrix} \right),\quad \; \alpha_k=\left( \begin{matrix}
0 &\sigma_k \\ \sigma_k &0 \\ \end{matrix}\right)  \qquad (k=1, 2, 3)\,,
$$
where the Pauli matrices are defined as
$$\sigma _1=\left( \begin{matrix} 0 & 1
\\ 1 & 0 \\ \end{matrix} \right),\quad  \sigma_2=\left( \begin{matrix} 0 & -i \\
i & 0 \\  \end{matrix}\right),\quad  \sigma_3=\left( 
\begin{matrix} 1 & 0\\  0 &-1\\  \end{matrix}\right) \, .$$
The operator $D_0$ is self-adjoint with domain $H^1(\R^3,\C^4)$ and its spectrum is $\Spec(D_0)=(-\ii,-1]\cup[1,\ii)$~\cite{Thaller}. 

\subsection{Distinguished self-adjoint extension}
The study of self-adjoint extensions is a classical subject for Dirac-Coulomb operators. For instance, the one-center Coulomb operator
$$D_0-\frac{\nu}{|x|}$$
is known to be essentially self-adjoint on $C^\ii_c(\R^3\setminus\{0\},\C^4)$ for $0\leq \nu\leq\sqrt3/2$, with domain $H^1(\R^3,\C^4)$ when $0\leq\nu<\sqrt{3}/2$, whereas it has several possible self-adjoint extensions for $\nu>\sqrt3/2$. When $\sqrt3/2<\nu<1$, there is a unique extension which is distinguished by the property that its domain is included in the `energy space' $H^{1/2}(\R^3,\C^4)$. The domain of the extension is always larger than $H^1(\R^3,\C^4)$ when $\nu\in[\sqrt3/2,1)$, with a regularity at the origin which deteriorates when $\nu$ increases. When $\nu=1$ one can also define a distinguished self-adjoint extension (obtained for instance by taking the limit $\nu\to1^-$) but its domain is no longer included in $H^{1/2}(\R^3,\C^4)$. There is no physically relevant extension for $\nu>1$. This corresponds to the previously mentioned property that we should work under the constraint $\nu=\alpha Z\le1$ which implies the bound $Z\leq 137$ on the maximal possible (integer) point charge in the periodic table, within Dirac theory. 

These relatively simple ODE-type results for the one-center Dirac operator have been generalized in many directions. Investigating how robust the distinguished extension is with regard to perturbations has indeed been the object of many works. A survey of known results, mainly in the one-center case, may be found for instance in~\cite[Sec.~1.3]{EstLewSer-19}. A typical result is that Dirac operators in the form
\begin{equation}
 D_V=D_0+V(x)\qquad\text{with}\qquad |V(x)|\leq \frac{\nu}{|x|},\qquad \nu<1
 \label{eq:pointwise_one-center}
\end{equation}
also have a unique distinguished self-adjoint extension, characterized by the property that the domain is included in $H^{1/2}(\R^3,\C^4)$. Hence a pointwise bound on $V$ is sufficient, which is rather remarkable for operators which are not semi-bounded. This extension can also be obtained as a norm-resolvent limit by truncating the singularity of the potential $V$. Note that the critical case $\nu=1$ was handled in~\cite{EstLos-07} for $V\geq0$ but the domain is not necessarily included in the energy space $H^{1/2}(\R^3,\C^4)$. 

There are fewer results about the multi-center case, which is however as physically important as the atomic case. The intuitive picture is that self-adjointness is essentially a local problem at the singularities of the potentials, hence the results should be similar for multi-center Coulomb potentials. Indeed, Nenciu \cite{Nenciu-77} and Klaus \cite{Klaus-80b} have proved that there is a unique distinguished self-adjoint extension for $D_0+V$ under the pointwise assumption that 
\begin{equation}
|V(x)|\leq \sum_{m=1}^M\frac{\nu_m}{|x-R_m|}
\label{eq:ass_V_subcritical}
 \end{equation}
with $R_m\neq R_\ell$ for $m\neq \ell$ and $0\leq \nu_m<1$ for all $m=1,...,M$. 
Note that we can add to $V$ any bounded potential (or even a regular potential in the sense of~\cite{Nenciu-76,Nenciu-77}), without changing the domain of self-adjointness. For other results on multi-center Dirac operators, see~\cite{Kanarski-85,BriHog-03}.

One important tool in these works is the Birman-Schwinger-type formula of the resolvent~\cite{Nenciu-76,KlaWus-79,Klaus-80b}
\begin{equation}
(D_0+V-z)^{-1}=(D_0-z)^{-1}-(D_0-z)^{-1}\sqrt{|V|}(1+SK_z)^{-1}S\sqrt{|V|}(D_0-z)^{-1} 
 \label{eq:resolvent}
\end{equation}
where 
\begin{equation}
K_z=\sqrt{|V|}(D_0-z)^{-1}\sqrt{|V|},
\label{eq:def_K}
\end{equation}
and $S={\rm sgn}(V)$. This formula is valid as long as $1-SK_z$ is invertible with bounded inverse, and it can serve to define $D_V$ via its resolvent. In the one-center case $|V(x)|\leq \nu|x|^{-1}$ we have for $z=0$
$$\norm{K_0}\leq \nu$$
(this was conjectured in~\cite{Nenciu-76} and then, proved in~\cite{Wust-77,Kato-83,ArrDuoVeg-13}). This gives the distinguished self-adjoint extension for $0\leq\nu<1$. In the multi-center case one cannot always use $z=0$ since it can be an eigenvalue, when $\sum_{j=1}^M\nu_j$ is large and the nuclei are close to each other. But the set of problematic $z$'s is at most countable, hence the formula also allows one to define the distinguished self-adjoint extension~\cite{Nenciu-77,Klaus-80b}.

The above results do not cover the case where $V(x)=-\mu\ast|x|^{-1}$ for general measures $\mu$. Such potentials indeed diverge like $\mu(\{R\})|x-R|^{-1}$ at points $R\in\R^3$ where $\mu(\{R\})>0$, but they can diverge at many other points in space where $\mu$ is not necessarily a delta. In this paper, we prove the following result, which confirms the intuition that only deltas are problematic with regard to self-adjointness.

\begin{thm}[Distinguished self-adjoint extension]\label{thm:distinguished}
Let $\mu$ be any finite signed Borel measure on $\R^3$, such that 
$$|\mu(\{R\})|<1\qquad \text{for all $R\in\R^3$.}$$ 
Then the operator
$$D_0-\mu\ast\frac{1}{|x|},$$
defined first on $H^1(\R^3,\C^4)$ or on $C^\ii_c(\R^3,\C^4)$, has a unique self-adjoint extension whose domain is included in $H^{1/2}(\R^3,\C^4)$. The functions in the domain $\cD(D_0-\mu\ast|x|^{-1})$ of the extension have a square-integrable gradient in $\R^3\setminus\bigcup_{j=1}^KB_r(R_j)$ 
for all $r>0$, where $R_1,...,R_K\in\R^3$ are all the points such that $|\mu(\{R_j\})|\geq1/2$. 
The operator $D_0-\mu\ast|x|^{-1}$ is the norm-resolvent limit of $D_0-\mu\ast\frac{1}{|x|}\1(|\mu\ast\frac{1}{|x|}|\leq n)$ when $n\to\ii$. Its essential spectrum is
$$\Spec_{\rm ess}\left(D_0-\mu\ast|x|^{-1}\right)=(-\ii,-1]\cup[1,+\ii).$$
\end{thm}

The proof of Theorem~\ref{thm:distinguished} is provided later in Section~\ref{sec:proof_distinguished}.

Note that in general $\cD(D_0-\mu\ast|x|^{-1})$ may differ from the domain of the operator $D_0-\sum_{j=1}^K\mu(\{R_j\})|x-R_j|^{-1}$, because the behavior of $\mu$ in the vicinity of the nuclei also plays a role. 

Although we think that there should be a similar result with a larger space than $H^{1/2}(\R^3,\C^4)$ under the weaker condition that $|\mu(\{R\})|\leq 1$, we have not investigated this question in the general setting. More about the critical case can be read in Section~\ref{sec:multicenter} where we investigate the particular case of a measure $\mu$ which is a pure sum of deltas, following~\cite{EstLewSer-19}.

One important argument of the proof is to show that the operator
$$\1_{B_R}\sqrt{\tilde\mu\ast\frac1{|x|}}\frac{1}{|p|^{\frac12}}$$
is compact, for every positive measure $\tilde\mu$ with no atom. Here we have used the notation $p=-i\nabla$ and $B_R$ for the ball of radius $R$ centered at the origin. Then, after separating the region about each nucleus from the rest of space, we show that 
$$\norm{\sqrt{|V_\mu|}\frac1{D_0+is}\sqrt{|V_\mu|}}<1$$
for $s$ large enough, where $V_\mu=\mu\ast|x|^{-1}$. This is what is needed to apply Nenciu's method~\cite[Cor.~2.1]{Nenciu-76}. 

\section{Domain and min-max formulas for positive measures}\label{sec:min-max}

By following a method introduced in~\cite{EstLos-07,EstLos-08} and further developed in~\cite{EstLewSer-19,SchSolTok-20}, one can describe the distinguished self-adjoint extension more precisely in the case of a positive measure: 
$$\mu\geq0.$$ 
What we really need in this section is that $V_\mu$ be bounded from below, but for simplicity we require the positivity of $\mu$ everywhere. We use the notation 
$$\boxed{V_\mu=\mu\ast\frac1{|x|}}$$
for the Coulomb potential induced by $\mu$. 

\subsection{Description of the domain}
Following~\cite{EstLewSer-19}, we introduce a new space for the upper component $\phi\in L^2(\R^3,\C^2)$ of a four-spinor. We define the following norm 
\begin{equation}
\norm{\phi}_{\cV_\mu}:= \left(\int_{\R^3}\frac{|\sigma\cdot \nabla\phi(x)|^2}{1+V_\mu(x)}\,dx+\int_{\R^3}|\phi(x)|^2\,dx\right)^{1/2}
\label{eq:norm_V}
\end{equation}
which is is well defined on $H^1(\R^3,\C^2)$ and controlled by the $H^1$--norm since $(1+V_\mu)^{-1}\leq1$. We can in fact replace $1+V_\mu$ by any $\lambda+V_\mu$ with $\lambda>0$ and we get an equivalent norm. Recall that $V_\mu\geq0$. 

Like in~\cite{EstLewSer-19} we need to know whether the completion $\cV_\mu$ of $H^1(\R^3,\C^2)$ for the norm in~\eqref{eq:norm_V} is the same as the largest space given by the conditions
$$\phi\in L^2(\R^3,\C^2),\qquad \frac{\sigma\cdot \nabla\phi}{(1+V_\mu)^{1/2}}\in L^2(\R^3,\C^2).$$
The following answers this question affirmatively.

\begin{thm}[The upper-spinor space $\cV_\mu$]\label{thm:def_cV_mu}
Let $\mu\geq0$ be any finite Borel measure on $\R^3$ so that 
$$\mu(\{R\})<1\qquad \text{for all $R\in\R^3$.}$$ 
We have 
\begin{equation}
\frac{\norm{\phi}^2_{H^{1/2}(\R^3,\C^2)}}{\max\big(2,16\,\mu(\R^3)\big)}\leq \norm{\phi}^2_{\cV_\mu}\leq \norm{\phi}^2_{H^{1}(\R^3,\C^2)}
\label{eq:in_H_1/2}
\end{equation}
for all $\phi\in H^1(\R^3,\C^2)$. 
The completion of $H^1(\R^3,\C^2)$ for the norm $\|\cdot\|_{\cV_\mu}$ is a subspace $\cV_\mu$ of $H^{1/2}(\R^3,\C^2)$ satisfying the continuous embeddings in~\eqref{eq:in_H_1/2}. It coincides with the completion of $C^\ii_c(\R^3,\C^2)$ for the same norm and is given by 
\begin{equation}
\cV_\mu=\left\{\phi\in L^2(\R^3,\C^2)\ :\ \exists g\in L^2(\R^3,\C^2),\quad \sigma\cdot \nabla\phi=(1+V_\mu)^{1/2}g\right\}
\label{eq:cV_mu_maximal}
\end{equation}
where $\sigma\cdot \nabla\phi$ is understood in the sense of distributions.
\end{thm}

The proof of Theorem~\ref{thm:def_cV_mu} is provided later in Section~\ref{sec:proof_cV_mu}. The first part of the theorem says that there is a Hardy-type inequality
\begin{equation}
\int_{\R^3}\frac{|\sigma\cdot \nabla\phi(x)|^2}{1+V_\mu(x)}\,dx+\int_{\R^3}|\phi(x)|^2\,dx\geq \frac{\norm{(-\Delta)^{\frac14}\phi}^2_{L^2(\R^3,\C^2)}}{\max\big(2,16\,\mu(\R^3)\big)},
\label{eq:Hardy-H-1/2}
\end{equation}
for $\phi\in H^1(\R^3,\C^2)$  and all positive measures $\mu$. The second part says that the space of functions $\phi\in L^2(\R^3,\C^2)$ such that $(1+V_\mu)^{-1/2}\sigma\cdot\nabla\phi\in L^2(\R^3,\C^2)$ (this being interpreted as in~\eqref{eq:cV_mu_maximal}), which could \emph{a priori} be larger than the completion $\cV_\mu$, is equal to $\cV_\mu$. This is an important property for what follows and it allows one to extend the inequality~\eqref{eq:Hardy-H-1/2} to all such functions~$\phi$. With regard to~\eqref{eq:cV_mu_maximal}, we remark that $V_\mu\in L^1_{\rm loc}(\R^3)$ for every Radon measure $\mu$, so that $(1+V_\mu)^{1/2}g\in L^1_{\rm loc}(\R^3,\C^2)$ is a distribution when $g\in L^2(\R^3,\C^2)$. Moreover, we prove later in Lemma~\ref{lem:V_mu} that $\nabla (1+V_\mu)^{-1/2}\in L^2(\R^3)$ for every $\mu$. This implies that $(1+V)^{-1/2}\sigma\cdot\nabla\phi$ makes sense as a distribution and, in~\eqref{eq:cV_mu_maximal}, it is then equivalent to requiring that this distribution belongs to $L^2(\R^3,\C^2)$. 

In the special case of 
$$\mu=\sum_{m=1}^M\nu_m\delta_{R_m}$$
with $0<\nu_m<1$, we have by~\cite{EstLewSer-19}
$$\cV_\mu=\left\{\phi \in L^2(\R^3,\C^2)\ :\ \int_{\R^3}\prod_{m=1}^M\frac{|x-R_m|}{1+|x-R_m|}|\sigma\cdot\nabla \phi(x)|^2\,dx<\ii\right\}.$$
In other words, the functions in $\cV_\mu$ must be in $H^1_{\rm loc}(\R^3\setminus\{R_1,...,R_M\},\C^2)$ and behave as stated close to the singularities. In general the space $\cV_\mu$ depends on the size and location of the singularities of the potential $V_\mu$, which are not necessarily produced by the atomic part of $\mu$. Recall that the potential $V_{\tilde\mu}$ of the non-atomic part $\tilde\mu$ of $\mu$ can still diverge in some places, namely at all the points $R\in\R^3$ so that 
$$\int_{\R^3}\frac{d\tilde\mu(x)}{|x-R|}=+\ii.$$ At each of these points, the norm is affected because $1/(1+V_\mu)$ tends to zero, allowing thereby $|\sigma\cdot\nabla\phi|^2$ to diverge a bit. 

We can now describe the domain of the distinguished self-adjoint extension using the space $\cV_\mu$.

\begin{thm}[Domain of the distinguished self-adjoint extension for $\mu\geq0$]\label{thm:self-adjoint}
Let $\mu\geq0$ be any finite Radon measure on $\R^3$ so that 
$$\mu(\{R\})<1\qquad \text{for all $R\in\R^3$.}$$ 
Then the domain of the self-adjoint extension from Theorem~\ref{thm:self-adjoint} is explicitly given  by
\begin{multline}
\cD(D_0-V_\mu)=\bigg\{\Psi=\begin{pmatrix}\phi\\\chi\end{pmatrix}\in L^2(\R^3,\C^4)\  :\\ \phi\in\cV_\mu,\quad  D_0\Psi-V_\mu\Psi\in L^2(\R^3,\C^4)\bigg\}
\label{eq:inclusion_domain_cV}
\end{multline}
where in the last condition $D_0\Psi$ and $V_\mu\Psi$ are understood in the sense of distributions. Moreover, this extension is the unique one included in $\cV_\mu\times L^2(\R^3,\C^2)$. We have 
$$\cD(D_0-V_\mu) \subset \cV_\mu\times \cV_\mu\subset H^{1/2}(\R^3,\C^4).$$

In addition, the Birman-Schwinger principle holds: $\lambda\in(-1,1)$ is an eigenvalue of $D_0-V_\mu$ if and only if $1$ is an eigenvalue of the bounded self-adjoint operator 
$$K_\lambda=\sqrt{V_\mu}\frac{1}{D_0-\lambda}\sqrt{V_\mu}.$$
\end{thm}

The condition $\mu\geq0$ is used to have $V_\mu\geq0$ which simplifies the definition of the space $\cV_\mu$. If $V_\mu$ is bounded-below, then the exact same result is valid with $1+V_\mu$ replaced by $C+V_\mu$ with a large enough constant $C$ everywhere. 

The proof of Theorem~\ref{thm:self-adjoint} is provided below in Section~\ref{sec:proofs_self_adjoint}.

\subsection{Min-max formulas for the eigenvalues}

Related to the above characterization of the domain are min-max formulas for eigenvalues~\cite{GriSie-99, DolEstSer-00, DolEstSer-00a, DolEstSer-03, DolEstSer-06,MorMul-15,Muller-16,EstLewSer-19,SchSolTok-20}. Our main result is the following

\begin{thm}[Min-max formulas]\label{thm:min-max}
Let $\mu\geq0$ be any finite non-trivial Radon measure on $\R^3$ so that 
$$\mu(\{R\})<1\qquad \text{for all $R\in\R^3$.}$$ 
Define the min-max values
\bq \lambda^{(k)} := \  \inf_{
 \scriptstyle W \ {\rm subspace \ of \ } F^+  \atop  \scriptstyle {\rm dim}
\ W =
k  } \  \Sup_{  \scriptstyle \Psi \in ( W \oplus F^- ) \setminus \{ 0 \} } \
\Frac{\pscal{\Psi, (D_0-V_\mu)\,\Psi}}{\|\Psi\|^2} \ ,
\qquad k \geq 1\,,
\label{eq:min-max} \eq
where $F$ is any chosen vector space satisfying 
$$\label{telesc}  C^\ii_c(\R^3,\C^4)\subseteq F \subseteq H^{1/2}(\R^3,\C^4)\,,$$  
and 
$$F^+:=\left\{\Psi=\begin{pmatrix}\phi\\0\\\end{pmatrix}\in F\right\},\qquad F^-:=\left\{\Psi=\begin{pmatrix}0\\\chi\\\end{pmatrix}\in F\right\}.$$
Then we have
\begin{itemize}
 \item[(i)] $\lambda^{(k)}$ is independent of the chosen space $F$;
 \item[(ii)] $\lambda^{(k)}\in[-1,1)$ for all $k$;
 \item[(iii)] it is a non-decreasing sequence converging to 1:
\begin{equation}
 \lim_{k\to\ii}\lambda^{(k)}=1.
 \label{eq:limit_lambda_k}
\end{equation} 
\end{itemize}
Let $k_0$ be the first integer so that 
$$\lambda^{(k_0)}>-1.$$
Then $(\lambda^{(k)})_{k\geq k_0}$ are all the eigenvalues of $D_0-V_\mu$ in non-decreasing order, repeated according to their multiplicity, which are larger than $-1$:
$$\Spec\left(D_0-\mu\ast\frac1{|x|}\right)\cap (-1,1)=\{\lambda^{(k_0)}, \lambda^{(k_0+1)},\cdots\}.$$
Finally, if 
\begin{equation}
 \mu(\R^3)\leq \frac{2}{\pi/2+2/\pi}\simeq 0.9,
 \label{eq:condition_Tix}
\end{equation}
then we have $\lambda^{(1)}\geq0$ and there is no eigenvalue in $(-1,0)$.
\end{thm}

The min-max formula~\eqref{eq:min-max} for the eigenvalues of $D_0-V_\mu$ is based on a decomposition of the four-component wavefunction into its upper and lower spinors, 
$$\Psi=\begin{pmatrix}\phi\\ \chi\\\end{pmatrix}.$$
That one can obtain the eigenvalues by maximizing over $\chi$ and minimizing over $\phi$ was suggested first in the Physics literature by Talman~\cite{Talman-86} and Datta-Devaiah~\cite{DatDev-88}. The min-max is largely based on the fact that the energy $\pscal{\Psi,(D_0-V_\mu)\Psi}$ is concave in $\chi$ and more or less convex in $\phi$ (up to finitely many directions corresponding to the indices $k<k_0$). This approach is reminiscent of the Schur complement formula, which is an important ingredient in the proof. Indeed, writing the eigenvalue equation in terms of $\phi$ and $\chi$, solving the one for $\chi$ and inserting it in the equation of $\phi$, one formally finds that 
\begin{equation}
\left(-\sigma\cdot \nabla\frac{1}{1+\lambda+V_\mu}\sigma\cdot \nabla+1-\lambda-V_\mu\right)\phi=0.
\label{eq:nonlinear_eigenvalue_equation}
\end{equation}
The formal operator on the left is associated with the quadratic form
\begin{equation}
 q_\lambda(\phi):=\int_{\R^3}\frac{|\sigma\cdot\nabla\phi|^2}{1+\lambda+V_\mu}\,dx+\int_{\R^3}(1-\lambda-V_\mu)|\phi|^2
 \label{eq:def_q_lambda}
\end{equation}
and most of the work is to show that it is bounded from below, and equivalent to the $\cV_\mu$--norm squared, up to addition of a constant. This allows one to give a meaning to the operator in~\eqref{eq:nonlinear_eigenvalue_equation} by means of the Riesz-Friedrichs method, hence to transform the (strongly indefinite) Dirac eigenvalue problem into an elliptic eigenvalue problem, nonlinear in the parameter $\lambda$. In the proof of Theorems~\ref{thm:self-adjoint} and~\ref{thm:min-max} we explain how to relate any information on the operator $K_z$ in~\eqref{eq:def_K} to that on the quadratic form $q_\lambda$ and we then directly apply~\cite{DolEstSer-00,EstLewSer-19,SchSolTok-20}. Note finally that condition~\eqref{eq:condition_Tix} is directly related to an inequality due to Tix~\cite{Tix-98} used in the proof.

\subsection{Application to (critical and sub-critical) multi-center potentials}\label{sec:multicenter}

Let us now discuss the special case of a positive measure made of a finite sum of deltas,
$$\mu=\sum_{m=1}^M\nu_m\delta_{R_m},\qquad 0<\nu_m\leq1,$$
which describes the nuclear density of a molecule. We always assume that the $R_m$ are all distinct from each other.

When $\bar\nu:=\max\nu_m<1$, Theorem~\ref{thm:distinguished} provides the self-adjointness of the corresponding multi-center Dirac-Coulomb operator. This was proved before in~\cite{Nenciu-77,Klaus-80b}. Theorem~\ref{thm:self-adjoint} gives the domain in terms of the space $\cV_\mu$ which, as we have already mentioned, is equal to
$$\cV_\mu=\left\{\phi \in L^2(\R^3,\C^2)\ :\ \int_{\R^3}\prod_{m=1}^M\frac{|x-R_m|}{1+|x-R_m|}|\sigma\cdot\nabla \phi(x)|^2\,dx<\ii\right\}.$$
This is proved by localizing about each nucleus. One can also give a simple explicit lower bound on the quadratic form $q_\lambda$ in terms of the $\nu_m$ and of the $R_m$, which remains valid in the critical case $\max\nu_m=1$.

\begin{lemma}[Lower bound on $q_\lambda$ for multi-center]\label{lem:estim_q_lambda}
Let $\mu=\sum_{m=1}^M\nu_m\delta_{R_m}$ with $0\leq \nu_m\leq 1$ and where the $R_m$ are all distinct. Then we have 
\begin{multline*}
q_\lambda(\phi)\geq (1-\bar\nu^2)\int_{\R^3}\frac{|\sigma\cdot \nabla \phi|^2}{2+V_\mu}\,dx\\
-\left(2\lambda+\frac{2(M-1)\bar\nu}{d}+\frac{\kappa}{d^2(1+\lambda)}\right)\int_{\R^3}|\phi|^2\,dx. 
\end{multline*}
for every $\phi\in H^1(\R^3,\C^2)$, where $\bar\nu=\max(\nu_m)$, $d=\max_{k\neq\ell}|R_k-R_\ell|$ and $\kappa>0$ is a universal constant. 
\end{lemma}

By arguing as in~\cite{EstLos-07,EstLos-08,EstLewSer-19}, this allows us to find a self-adjoint extension distinguished from the property that its domain satisfies 
$$\cD(D_0-V_\mu)\subset \cW_\mu\times L^2(\R^3,\C^2)$$
where $\cW_\mu$ is the space obtained after closing the quadratic form $q_\lambda$. This space is larger than $\cV_\mu$ when $\max(\nu_m)=1$. A simple localization as in the proof of Lemma~\ref{lem:estim_q_lambda} allows to apply the results of~\cite{EstLewSer-19} and deduce that 
\begin{equation}
\cW_\mu=\left\{\phi \in L^2(\R^3,\C^2)\ :\ \int_{\R^3}\frac{\Big|\sigma\cdot\nabla \prod_{m=1}^M|x-R_m| \phi(x)\Big|^2}{\prod_{m=1}^M|x-R_m|\big(1+|x-R_m|\big)}\,dx<\ii\right\}.
\end{equation}
Arguing like in~\cite{EstLewSer-19}, one can prove that the distinguished self-adjoint extension is the norm-resolvent limit of the operators obtained after truncating the potential or after decreasing the critical nuclear charges. One can indeed treat any potential $V$ so that 
$$0\leq V\leq V_\mu$$
but then the space $\cW_\mu$ has to be modified accordingly. The arguments are exactly the same as in~\cite{EstLewSer-19}.


In chemistry one is interested in the \emph{potential energy surface} which, by definition, is the graph of the first eigenvalue of the multi-center Dirac-Coulomb operator, seen as a function of the locations of the nuclei at fixed $\nu_m$ and including the nuclear repulsion:
$$(R_1,...,R_M)\mapsto \lambda_1\left(D_0-\sum_{m=1}^M\frac{\nu_m}{|x-R_m|}\right)+\sum_{1\leq m<\ell\leq M}\frac{\nu_m\nu_\ell}{|R_m-R_\ell|}.$$
The following is an extension of similar results proved before for $M=2$ in~\cite{Klaus-80b,HarKla-83,BriHog-03}.

\begin{thm}[Molecular case]\label{Thm-continuity}
Assume that $R_1, \dots, R_M$ are $M$ distinct points in $\R^3$, and that $\mu=\sum_{m=1}^M\nu_m\delta_{R_m}$ with $0< \nu_m< 1$.
Let 
$$\lambda_1\left(D_0-\sum_{m=1}^M\frac{\nu_m}{|x-R_m|}\right)$$
be the first min-max level as in~\eqref{eq:min-max} . 
Then, 

\noindent $(i)$ the map $(R_1,...,R_M)\mapsto \lambda_1\left(D_0-\sum_{m=1}^M\frac{\nu_m}{|x-R_m|}\right)$ is a continuous function on the open set $\Omega$ defined as
\begin{equation*}
\Omega=\bigg\{(R_1,...,R_M)\in(\R^3)^M\ :\ \lambda_1\left(D_0-\sum_{m=1}^M\frac{\nu_m}{|x-R_m|}\right)>-1\bigg\}.
\end{equation*}

\medskip

\noindent $(ii)$ Moreover,
\begin{equation}
\lim_{\min_{k\neq \ell}|R_k-R_\ell|\to\ii}\lambda_1\left(D_0-\sum_{m=1}^M\frac{\nu_m}{|x-R_m|}\right)=\sqrt{1-\max\nu_m^2}.
\end{equation} 

\medskip

\noindent $(iii)$ If in addition $\sum_{m=1}^M\nu_m< 2(\pi/2+2/\pi)^{-1}$ then 
\begin{equation}
\lim_{\max_{k\neq \ell}|R_k-R_\ell|\to0}\lambda_1\left(D_0-\sum_{m=1}^M\frac{\nu_m}{|x-R_m|}\right)=\sqrt{1-\left(\sum_{m=1}^M\nu_m\right)^2}.
\end{equation}
\end{thm}

The part $(iii)$ of the theorem actually holds for $\sum_{m=1}^M\nu_m<\overline{\nu}_1$ where $\overline{\nu}_1$ is a constant defined in the second part~\cite{EstLewSer-21b} of this work (denoted there simply by $\nu_1$). If $\overline\nu_1=1$ as we believe, then $(iii)$ holds for all $\sum_{m=1}^M\nu_m<1$. In~\cite{BriHog-03} the result is claimed for $M=2$ and $\nu_1=\nu_2<1/2$ but we could not fill all the details of the argument of the proof of~\cite[Lem.~3.1]{BriHog-03}. 

For $M=2$ it is a famous conjecture that the energy of a diatomic molecule is always greater than the single atom with the two nuclei merged. This property was conjectured for two-atoms Dirac operators by Klaus in~\cite[p.~478]{Klaus-80b} and by Briet-Hogreve in~\cite[Sec.~2.4]{BriHog-03}. Numerical simulations from~\cite{ArtSurIndPluSto-10,McConnell-13} seem to confirm the conjecture for $M=2$, even for large values of the nuclear charges. We make the stronger conjecture that the same holds for any $M$.

\begin{conjecture}\label{conjecture_multicenter}
We have
\begin{equation}
\lambda_1\left(D_0-\sum_{m=1}^M\frac{\nu_m}{|x-R_m|}\right)\geq \lambda_1\left(D_0-\frac{\sum_{m=1}^M\nu_m}{|x|}\right)=\sqrt{1-\left(\sum_{m=1}^M\nu_m\right)^2},
\end{equation}
for all $M\geq2$, all $R_1,...,R_M\in\R^3$ and all $\nu_m\geq0$ so that $\sum_{m=1}^M\nu_m\leq1$.
\end{conjecture}

In~\cite{EstLewSer-21b} we discuss a stronger conjecture which implies Conjecture~\ref{conjecture_multicenter}. Note that Conjecture~\ref{conjecture_multicenter} has been proved in the non-relativistic case, as recalled in the introduction.


\section{Proof of Theorem~\ref{thm:distinguished}}\label{sec:proof_distinguished}

The proof relies on two preliminary lemmas which we first state and show, before we turn to the actual proof of the theorem. 

Loosely speaking, the first lemma asserts that $g(x)|p|^{-s}f(x)$ is compact when $f$ and $g$ have disjoint supports. Recall that everywhere $p=-i\nabla$. 

\begin{lemma}[Compactness for disjoint supports]\label{lem:compacness-disjoint-supports}
Let $f\in L^2(\R^d)$ with support in a compact set $B\subset\R^d$. Let $\Omega$ be a (bounded or unbounded) set in $\R^d$ so that ${\rm d}(B,\Omega)>0$ and let $g\in L^2_{\rm loc}(\R^d)$ supported on $\Omega$, such that 
$$\int_{\Omega}\frac{|g(x)|^2}{(1+|x|)^{2(d-s)}}\,dx<\ii,$$
where $0<s<d$.
Then the operator
$$K=g(x)\frac{1}{|p|^s}f(x)$$
is compact and its norm can be estimated by
\begin{equation}
 \norm{K}\leq C\norm{f}_{L^2(B)}\left(\int_{\Omega}\frac{|g(x)|^2}{|x|^{2(d-s)}}\,dx\right)^{\frac12},
 \label{eq:bound_norm_K}
\end{equation}
where $C$ only depends on $s$, on the dimension, on ${\rm d}(B,\Omega)$ and on $\sup_{x\in B}|x|$. 
\end{lemma}

\begin{proof}
The kernel of the operator $K$ is given by
$$K(x,y)=\kappa\frac{\1_\Omega(x)g(x)f(y)}{|x-y|^{d-s}}$$
for some constant $\kappa$. Since the two functions $f$ and $g$ have supports at a finite distance from each other, we have $|x-y|\geq {\rm d}(B,\Omega)>0$ and the denominator never vanishes. Since $B$ is compact, we even have $|x-y|\geq |x|-R$ where $R=\max_{y\in B}|y|$. In particular, we conclude that $|x-y|\geq c(|x|+1)$ for all $x\in \Omega$ and all $y\in B$. Therefore, the kernel of $K$ is pointwise bounded by
$$|K(x,y)|\leq \frac{\kappa}{c^{d-s}}\frac{|g(x)|}{(1+|x|)^{d-s}}|f(y)|.$$
The right side is a rank-one operator which is bounded under the condition that $f\in L^2$ and $g(1+|x|)^{s-d}\in L^2$. This shows that $K$ is bounded as in~\eqref{eq:bound_norm_K}. 

To prove the compactness of $K$ we can approximate $f$ and $g$ by functions in $C^\ii_c$ and use classical compactness results. But we can also argue directly as follows. Let $u_n\wto0$ be any sequence converging weakly to 0 with $\|u_n\|_{L^2}=1$. We remark that 
$$(Ku_n)(x)=\kappa\;g(x)\int_{B}\frac{f(y)u_n(y)}{|x-y|^{d-s}}\,dy$$
converges to 0 almost everywhere, since $y\mapsto f(y)|x-y|^{s-d}$ belongs to $L^2$ for all $x\in\Omega$. In addition, we have the pointwise bound
$$|(Ku_n)(x)|\leq \frac{\kappa}{c^{d-s}}\frac{|g(x)|}{(1+|x|)^{d-s}}\norm{f}_{L^2(B)}$$
which, by the dominated convergence theorem, implies that $\|Ku_n\|\to0$. 
\end{proof}

Using Lemma~\ref{lem:compacness-disjoint-supports} we can show the following result.

\begin{lemma}[Local compactness in the absence of atoms]\label{lem:compactness-no-atom}
Let $\tilde\mu\geq0$ be a finite Radon measure on $\R^3$, with no atom. Then 
$$\1_{B_R}\sqrt{\tilde\mu\ast\frac1{|x|}}\frac{1}{|p|^{\frac12}}$$
is a compact operator for every finite $R>0$. 
\end{lemma}

Note that the operator $\sqrt{V_{\tilde\mu}}|p|^{-1/2}$ is not compact since at infinity it essentially behaves like $\sqrt{\tilde\mu(\R^3)}|x|^{-1/2}|p|^{-1/2}$ which is not compact. The characteristic function $\1_{B_R}$ is really necessary. In addition, the corresponding result cannot hold for a measure which comprises some deltas, due to the lack of compactness at the corresponding centers. 

\begin{proof}[Proof of Lemma~\ref{lem:compactness-no-atom}]
First we write $\tilde\mu=\tilde\mu \1_{B_N}+\tilde\mu \1_{\R^3\setminus B_N}$  where $B_N$ is the ball of radius $N$, and remark that 
\begin{multline}
\norm{\left(\sqrt{\tilde\mu\ast\frac1{|x|}}-\sqrt{(\tilde\mu\1_{B_N}) \ast\frac1{|x|}}\right)\frac{1}{|p|^{\frac12}}}\\
\leq \norm{\sqrt{(\tilde\mu\1_{\R^3\setminus B_N})\ast\frac1{|x|}}\frac{1}{|p|^{\frac12}}}\leq \sqrt{\frac\pi2 \tilde\mu(\R^3\setminus B_N)}. 
\label{eq:cut_mu_tilde}
\end{multline}
The first inequality holds because $\sqrt{V_{\mu}}-\sqrt{V_{\mu_1}}\leq \sqrt{V_{\mu_2}}$ pointwise, for $\mu=\mu_1+\mu_2$. The second uses Kato's inequality
\begin{equation}
 \frac{1}{|x|}\leq\frac\pi2 |p|
 \label{eq:Kato}
\end{equation}
which implies that 
$$\norm{\sqrt{\mu\ast\frac1{|x|}}\frac1{|p|^{\frac12}}}=\norm{\frac1{|p|^{\frac12}}\sqrt{\mu\ast\frac1{|x|}}}\leq \sqrt{\frac\pi2}\sqrt{\mu(\R^3)}$$
for every bounded measure $\mu$. The right side of~\eqref{eq:cut_mu_tilde} tends to zero when $N\to\ii$ and this shows that we may assume for the rest of the proof that $\tilde\mu$ has compact support. 

For $r>0$, let us consider two tilings of the whole space $\R^3$ with cubes $C_j=3r(j+[-1/2,1/2)^3)$ and $C'_k=r(k+[-1/2,1/2)^3)$ of side length $3r$ and $r$, respectively, where $j,k\in\Z^3$. For every $k$, we call $j_k$ the index of the large cube $C_{j_k}$ of which $C'_k$ is exactly at the center. Let $\eps>0$. By compactness of the support of $\tilde\mu$ and the fact that it has no atom, we can find $r>0$ so that $\tilde\mu(C_j)\leq\eps$ for every $j$. 

We then write
$$\frac1{|p|^{\frac12}}\1_{B_R}\left(\tilde\mu\ast\frac1{|x|}\right)\frac1{|p|^{\frac12}}=\sum_{k}\frac1{|p|^{\frac12}}\1_{B_R\cap C'_k}\left(\tilde\mu(\1_{C_{j_k}}+\1_{\R^3\setminus C_{j_k}})\ast\frac1{|x|}\right)\frac1{|p|^{\frac12}}.$$
Recall that the sum is finite.
The sets $C_k'$ and $\R^3\setminus C_{j_k}$ are at a distance at least equal to $r$ from each other. Then 
$$\left|\1_{C'_k}\left(\tilde\mu\1_{\R^3\setminus C_{j_k}}\ast\frac1{|x|}\right)\right|\leq\frac{c}{r}.$$
In particular this function is in $L^p(C'_k)$ for all $1\leq p\leq\ii$ and the operator 
$$\frac1{|p|^{\frac12}}\1_{B_R\cap C'_k}\left(\tilde\mu\1_{\R^3\setminus C_{j_k}}\ast\frac1{|x|}\right)\frac1{|p|^{\frac12}}$$
is compact. This is true for instance because $|p|^{-1/2}f(x)|p|^{-1/2}$ is compact under the condition that $f\in L^3(\R^3)$, by Cwikel's inequality~\cite{Simon-79}. Hence, at this step we have written
$$\frac1{|p|^{\frac12}}\1_{B_R}\left(\tilde\mu\ast\frac1{|x|}\right)\frac1{|p|^{\frac12}}=\sum_{k}\frac1{|p|^{\frac12}}\1_{B_R\cap C'_{k}}\left(\tilde\mu\1_{C_{j_k}}\ast\frac1{|x|}\right)\frac1{|p|^{\frac12}}+K_1$$
where $K_1$ is compact. 

Next, for every $k$ we insert another localization as follows
\begin{multline*}
\frac1{|p|^{\frac12}}\1_{C'_{k}}\left(\tilde\mu\1_{C_{j_k}}\ast\frac1{|x|}\right)\frac1{|p|^{\frac12}}\\=(\1_{C_{j_k}}+\1_{\R^3\setminus C_{j_k}})\frac1{|p|^{\frac12}}\1_{C'_{k}}\left(\tilde\mu\1_{C_{j_k}}\ast\frac1{|x|}\right)\frac1{|p|^{\frac12}} (\1_{C_{j_k}}+\1_{\R^3\setminus C_{j_k}}).
\end{multline*}
The operator 
$$\1_{\R^3\setminus C_{j_k}}\frac1{|p|^{\frac12}}\1_{C'_k}\sqrt{\tilde\mu\1_{C_{j_k}}\ast\frac1{|x|}}$$
is compact, by Lemma~\ref{lem:compacness-disjoint-supports}. Indeed, we have 
$$\1_{C'_k}\left(\tilde\mu\1_{C_{j_k}}\ast\frac1{|x|}\right)=\1_{C'_k}\left(\tilde\mu\1_{C_{j_k}}\ast\frac{\1_{C_{j_k}+C'_k}}{|x|}\right)\in L^1(C'_k)$$
so that its square root is in $L^2(C'_k)$, and
$$\int_{\R^3\setminus C_{j_k}}\frac{dx}{(1+|x|)^{5}}<\ii.$$
This proves that 
$$\frac1{|p|^{\frac12}}\1_{B_R}\left(\tilde\mu\ast\frac1{|x|}\right)\frac1{|p|^{\frac12}}=\sum_{k}\1_{C_{j_k}}\frac1{|p|^{\frac12}}\1_{B_R\cap C'_k}\left(\tilde\mu\1_{C_{j_k}}\ast\frac1{|x|}\right)\frac1{|p|^{\frac12}}\1_{C_{j_k}}+K_2$$
where $K_2$ is compact. By Kato's inequality~\eqref{eq:Kato} the first operator is bounded above as follows:
\begin{multline*}
\sum_{k}\1_{C_{j_k}}\frac1{|p|^{\frac12}}\1_{B_R\cap C'_k}\left(\tilde\mu\1_{C_{j_k}}\ast\frac1{|x|}\right)\frac1{|p|^{\frac12}}\1_{C_{j_k}}\\
\leq \frac\pi2 \sum_{\substack{k\ :\\ C_k\cap B_R\neq\emptyset}}\1_{C_{j_k}}\tilde\mu(C_{j_k})\leq \eps\frac\pi2  \sum_{\substack{k\ :\\ C_k\cap B_R\neq\emptyset}}\1_{C_{j_k}}\leq \frac{27\pi}2 \eps. 
\end{multline*}
Here we have used that $\tilde\mu(C_j)\leq\eps$ for every $j$ by our choice of $r$. Therefore our initial operator is the norm-limit of a sequence of compact operators and it must be compact. 
\end{proof}

Now we can provide the

\begin{proof}[Proof of Theorem~\ref{thm:distinguished}]
Our goal is to show that 
\begin{equation}
\limsup_{|s|\to\ii} \norm{\sqrt{|V_\mu|}\frac{1}{D_0+is}\sqrt{|V_\mu|}}\leq \max_{R\in\R^3}|\mu(\{R\})|<1,
 \label{eq:to_be_estimated}
\end{equation}
where we recall that $V_\mu:=\mu\ast|x|^{-1}$. We write $\mu$ in the form
$$\mu=\sum_{m=1}^\ii\nu_m\delta_{R_m}+\tilde\mu$$
where it is understood that the $R_m$ are all distinct and where $\max_m|\nu_m|<1$ by assumption. The signed measure $\tilde\mu$ has no atom. Here we allow infinitely many singularities for simplicity of notation but many of the $\nu_m$ could vanish. We know that $\sum_{m=1}^\ii|\nu_m|<\ii$. 

Similarly as in~\eqref{eq:cut_mu_tilde} we can first write
\begin{equation}
 \mu=\left(\sum_{m=1}^K\nu_m\delta_{R_m}+\tilde\mu\1_{B_N}\right)+\left(\sum_{m\geq K+1}\nu_m\delta_{R_m}+\tilde\mu\1_{\R^3\setminus B_N}\right):=\mu_1+\mu_2.
 \label{eq:decompose_mu}
\end{equation}
Using Kato's inequality~\eqref{eq:Kato} and the fact that $|\mu_2|(\R^3)$ is small for $K$ and $N$ large enough, we see that it suffices to show the limit~\eqref{eq:to_be_estimated} for $\mu$ having finitely many atoms and for $\tilde\mu$ of compact support, which we assume for the rest of the proof. In this case we have gained that 
$$|V_{\mu}(x)|=\left|\int_{\R^3}\frac{d\mu(y)}{|x-y|}\right|\leq \frac{|\mu|(\R^3)}{|x|-N}$$
where ${\rm supp}(\mu)\subset B_N$. In particular, $V_\mu$ is bounded at infinity. 

We then write the potential $V_\mu$ in the form
$$V_\mu=\sum_{m=1}^K\frac{\nu_m\1_{B_\eta(R_m)}}{|x-R_m|}+\1_{B_R}V_{\tilde\mu}+\sum_{m=1}^K\frac{\nu_m\1_{\R^3\setminus B_\eta(R_m)}}{|x-R_m|}+\1_{\R^3\setminus B_R}V_{\tilde\mu}$$
where $B_\eta(R_m)$ is the ball of radius $\eta$ centered at $R_m$. We choose $\eta<\min_{1\leq m\neq \ell\leq K}|R_m-R_\ell|/2$ (half of the smallest distance between the nuclei). 
Using that $\sqrt{\sum V_j}\leq \sum\sqrt{V_j}$, this shows that 
\begin{equation}
|V_\mu|^{\frac12}\leq \sum_{m=1}^K\frac{\sqrt{|\nu_m|}\1_{B_\eta(R_m)}}{|x-R_m|^{\frac12}}+\1_{B_R}\sqrt{V_{|\tilde\mu|}}+\sqrt{\sum_{m=1}^K\frac{|\nu_m|}{\eta}}+\sqrt{\frac{|\tilde\mu|(\R^3)}{R-N}},
\label{eq:decomp_sqrt_V_mu}
\end{equation}
because $R$ was chosen larger than $N$. Next we replace $\sqrt{|V_\mu|}$ by the function on the right in 
$$\norm{\sqrt{|V_\mu|}\frac{1}{D_0+is}\sqrt{|V_\mu|}},$$
expand everything and estimate all the terms separately.

The dominant terms are the ones close to the singularities
$$\sum_{m=1}^K|\nu_m|\frac{\1_{B_\eta(R_m)}}{|x-R_m|^{\frac12}}\frac{1}{D_0+is}\frac{\1_{B_\eta(R_m)}}{|x-R_m|^{\frac12}}.$$
It is time to recall that for Coulomb potentials 
\begin{multline}
\Spec\left(\frac{1}{|x|^{\frac12}}\frac{1}{\alpha\cdot p+\beta}\frac{1}{|x|^{\frac12}}\right)=\Spec_{\rm ess}\left(\frac{1}{|x|^{\frac12}}\frac{1}{\alpha\cdot p+\beta}\frac{1}{|x|^{\frac12}}\right)\\=\Spec\left(\frac{1}{|x|^{\frac12}}\frac{1}{\alpha\cdot p}\frac{1}{|x|^{\frac12}}\right)=\Spec_{\rm ess}\left(\frac{1}{|x|^{\frac12}}\frac{1}{\alpha\cdot p}\frac{1}{|x|^{\frac12}}\right)=[-1,1] 
\label{eq:rappel_Coulomb}
\end{multline}
and that 
\begin{equation}
\norm{\frac{1}{|x|^{\frac12}}\frac{1}{\alpha\cdot p+is}\frac{1}{|x|^{\frac12}}}=\norm{\frac{1}{|x|^{\frac12}}\frac{1}{\alpha\cdot p+\beta+is}\frac{1}{|x|^{\frac12}}}=1
\label{eq:rappel_Coulomb_2}
\end{equation}
for all $s\in\R$. See~\cite{Nenciu-76,Wust-77,Klaus-80b,Kato-83,ArrDuoVeg-13}. This implies immediately that 
\begin{multline*}
\norm{\sum_{m=1}^K|\nu_m|\frac{\1_{B_\eta(R_m)}}{|x-R_m|^{\frac12}}\frac{1}{D_0+is}\frac{\1_{B_\eta(R_m)}}{|x-R_m|^{\frac12}}}\\
\leq\max_m|\nu_m|\norm{\sum_{m=1}^K\1_{B_\eta(R_m)}}=\max_m|\nu_m|<1.
\end{multline*}
Note that the estimate does not depend on $K$ but requires $\eta$ to be small enough to guarantee that the balls do not overlap. The choice of $\eta$ depends on the smallest distance between the nuclei.

All the other terms are going to be small for $s$ large enough. First, using that 
$$\frac{|p|^{\frac12}}{|D_0+is|}=\frac{|p|^{\frac12}}{\sqrt{p^2+1+s^2}}\leq\frac{1}{\sqrt{2|s|}}$$
we see that all the terms involving the constant potential 
$$\sqrt{\sum_{m=1}^K\frac{|\nu_m|}{\eta}}+\sqrt{\frac{|\tilde\mu|(\R^3)}{R-N}}$$
in~\eqref{eq:decomp_sqrt_V_mu} have a norm of the order $O(|s|^{-1/2})$. The terms involving $\1_{B_R}\sqrt{V_{\tilde\mu}}$ can all be written in the form
$$\sqrt{|V_{\mu'}|}\frac1{|p|^{\frac12}}\frac{|p|}{D_0+is}\frac1{|p|^{\frac12}}\sqrt{|V_{\tilde\mu}|}\1_{B_R}$$
for some measure $\mu'$. 
They all tend to 0 in norm when $s\to0$. This is because $BA_sK$ tends to zero in norm when $B$ is bounded, $K$ is compact and $A_s\to0$ strongly with $\|A_s\|\leq C$ uniformly in $s$. The operator $|p|^{-1/2}\sqrt{|V_{\tilde\mu}|}\1_{B_R}$ is compact by Lemma~\ref{lem:compactness-no-atom}.

We are left with the more complicated interaction terms between the balls
\begin{equation}
\frac{\1_{B_\eta(R_k)}}{|x-R_k|^{\frac12}}\frac{1}{D_0+is}\frac{\1_{B_\eta(R_m)}}{|x-R_m|^{\frac12}}
\label{eq:interaction_two_balls}
\end{equation}
with $k\neq m$. We have
$$\frac{1}{D_0+is}=\frac{\alpha\cdot p}{p^2+1+s^2}+\frac{\beta}{p^2+1+s^2}-i\frac{s}{p^2+1+s^2}$$
and the term involving $\beta$ is easily controlled by $|p|^{-1}(1+s^2)^{-1/2}$ and Kato's inequality. Like in Lemma~\ref{lem:compacness-disjoint-supports}, our idea to deal with the other two terms is to use \emph{pointwise} kernel bounds. Note that operator bounds are not very useful here since we have different functions on both sides of $(D_0+is)^{-1}$. Recall also that $|A(x,y)|\leq B(x,y)$ implies $\|A\|\leq\|B\|$. The kernel of the first operator is
\begin{multline*}
\frac{\alpha\cdot p}{p^2+1+s^2}(x,y)=i\frac{\alpha\cdot(x-y)}{4\pi|x-y|^3}e^{-\sqrt{1+s^2}|x-y|}\\+i\sqrt{1+s^2}\frac{\alpha\cdot(x-y)}{4\pi|x-y|^2}e^{-\sqrt{1+s^2}|x-y|}
\end{multline*}
and it can be bounded by
$$\left|\frac{\alpha\cdot p}{p^2+1+s^2}(x,y)\right|\leq \frac{e^{-\sqrt{1+s^2}|x-y|}}{4\pi|x-y|^2}+\sqrt{1+s^2}\frac{e^{-\sqrt{1+s^2}|x-y|}}{4\pi|x-y|}\leq \frac{C}{|s|^{\frac12}|x-y|^{\frac52}}.$$
Similarly, we have
$$\left|\frac{is}{p^2+1+s^2}(x,y)\right|=\frac{|s|\;e^{-\sqrt{1+s^2}|x-y|}}{4\pi|x-y|}\leq \frac{C}{|s|^{\frac12}|x-y|^{\frac52}}.$$
We obtain that 
\begin{multline}
\norm{\frac{\1_{B_\eta(R_k)}}{|x-R_k|^{\frac12}}\frac{1}{D_0+is}\frac{\1_{B_\eta(R_m)}}{|x-R_m|^{\frac12}}}\\
\leq \frac{C}{|s|^{\frac12}}\norm{\frac{\1_{B_\eta(R_k)}}{|x-R_k|^{\frac12}}\frac{1}{|p|^{\frac12}}\frac{\1_{B_\eta(R_m)}}{|x-R_m|^{\frac12}}}+\frac{C}{|s|^{\frac12}}.
\end{multline}
Using Lemma~\ref{lem:compacness-disjoint-supports} the norm on the right is finite and we conclude that the interactions between balls are a $O(|s|^{-1/2})$. This concludes the proof of~\eqref{eq:to_be_estimated}.

By the result of Nenciu~\cite[Cor.~2.1]{Nenciu-76} (see also Klaus~\cite{Klaus-80b}) based on the resolvent expansion~\eqref{eq:resolvent}, the estimate~\eqref{eq:to_be_estimated} proves that $D_0-V_\mu$ has a unique self-adjoint extension whose domain is included in $H^{1/2}(\R^3)$. In addition, $D_0-V_\mu\1(|V_\mu|\leq n)$ converges to $D_0-V_\mu$ in the norm-resolvent sense when $n\to\ii$~\cite{KlaWus-79,Klaus-80b}. That the essential spectrum is equal to $(-\ii,-1]\cup[1,\ii)$ is also a consequence of the resolvent formula~\eqref{eq:resolvent} as in~\cite{Nenciu-76,KlaWus-79}. This is because
$$\frac{1}{D_0-z}\sqrt{|V_\mu|}$$
is compact for $z\in\C\setminus\sigma(D_0)$. This follows from the Hausdorff-Young inequality and the fact that $\sqrt{|V_\mu|}\in (L^{6-\eps}+L^{6+\eps})(\R^3)$ whereas $(\alpha\cdot p+\beta -z)^{-1}\in L^r(\R^3)$ for all $r>3$, hence belongs to $(L^{6-\eps}\cap L^{6+\eps})(\R^3)$. 

It remains to prove the statement about the gradient of functions in the domain. We call $R_1,...,R_K$ all the points such that $|\mu(\{R_k\})|\geq1/2$ and deduce that for every $x_0\in\R^3\setminus\{R_1,...,R_K\}$, we have 
$$\lim_{r\to0}|\mu|(B_r(x_0))<\frac12$$
where $B_r(x_0)$ is the ball of radius $r$ centered at $x_0$. Let $\chi$ be a smooth function supported on $B_{1/2}(0)$ and let $\chi_r(x)=\chi((x-x_0)/r)$. Every $\Psi$ in the domain of $D_0-\mu\ast|x|^{-1}$ satisfies
$$D_0\Psi-V_\mu\Psi=\Phi\in L^2(\R^3)$$
in $H^{-1/2}$, so that 
$$D_0(\chi_r\Psi)-V_\mu\chi_r\Psi=\chi_r\Phi-i(\alpha\cdot\nabla\chi_r)\Psi\in L^2(\R^3).$$
We decompose $\mu=\mu\1_{B_r}+\mu\1_{\R^3\setminus B_r}$ and use that 
$$\left|V_{\mu\1_{\R^3\setminus B_r}}\right|\leq \frac{2|\mu|(\R^3)}{r}\qquad\text{on $B_{r/2}$.}$$
This gives
$$\left(D_0-V_{\mu\1_{B_r}}\right)\chi_r\Psi\in L^2(\R^3).$$
For $\mu(B_r)<1/2$, the operator $D_0-V_{\mu\1_{B_r}}$ is self-adjoint on $H^1(\R^3)$ by Hardy's inequality. This proves, as stated, that $\chi_r\Psi\in H^1(\R^3)$. Using that $\mu(\R^3\setminus B_R)\to0$ when $R\to\ii$ we can prove in a similar manner that $(1-\chi_R)\Psi\in H^1(\R^3)$. We obtain the claim by covering $\R^3\setminus\cup_{j=1}^KB_r(R_j)$ with finitely many balls together with the complement of a large ball. This concludes the proof of Theorem~\ref{thm:distinguished}.
\end{proof}

\section{Proof of Theorem~\ref{thm:def_cV_mu}}\label{sec:proof_cV_mu}

\subsubsection*{Step 1. Proof of the estimate~\eqref{eq:in_H_1/2} on $\|\cdot\|_{\cV_\mu}$}
The upper bound in~\eqref{eq:in_H_1/2} follows immediately from the fact that $(1+V_\mu)^{-1}\leq1$ and we concentrate on proving the lower bound. Let $\mu$ be a finite non-negative measure on $\R^3$ and $V_\mu:=\mu\ast|\cdot|^{-1}$. Then, by Hardy's inequality $|x|^{-2}\leq 4(-\Delta)\leq 4(D_0)^2$, we have 
$$\norm{V_\mu\frac{1}{D_0}}\leq 2\mu(\R^3).$$
We also have 
$$\norm{|p|(\beta+1)\frac1{D_0}}\leq \sup_{p\in\R^3}\frac{2|p|}{\sqrt{1+|p|^2}}\leq\sqrt2.$$
By the Rellich-Kato theorem, this proves that the operator
$$D_0-\frac{V_\mu}{8\mu(\R^3)}-|p|\frac{\beta+1}{4}$$
is self-adjoint on $H^1(\R^3)$ and that $0$ is not in its spectrum, with a universal estimate on the gap around the origin.  For the same reason, the operator 
$$D_0-\frac{tV_\mu}{8\mu(\R^3)}-|p|\frac{\beta+1}{4}$$
has a gap around the origin at least as big as when $t=1$, for all $t\in[0,1]$. Note that $|p|(\beta+1)/4$ only acts on the upper spinor. Restricted to lower spinors, the quadratic form associated with this operator is just $-1-tV_\mu/(8\mu(R^3))\leq -1$. From the min-max principle and a continuation argument in $t$ from~\cite{DolEstSer-00}, the fact that $0$ is never in the spectrum is equivalent to saying that 
\begin{multline*}
\int_{\R^3}\frac{|\sigma\cdot\nabla\phi(x)|^2}{1+\frac{V_\mu(x)}{8\mu(\R^3)}}\,dx\geq \frac{1}{8\mu(\R^3)}\int_{\R^3}V_\mu(x)|\phi(x)|^2\,dx\\
+\frac12\pscal{\phi,|p|\phi}-\frac12\norm{\phi}_{L^2(\R^3)}^2
\end{multline*}
for all $\phi\in H^1(\R^3,\C^2)$ and all non-negative finite measures $\mu$ over $\R^3$. Dropping the potential term and using the inequality
$$\frac{1}{1+a}\geq \frac{1}{\max(1,b)\left(1+\frac{a}{b}\right)}$$
gives
\begin{equation}
\int_{\R^3}\frac{|\sigma\cdot\nabla\phi(x)|^2}{1+V_\mu(x)}\,dx\geq \frac{\pscal{\phi,|p|\phi}-\norm{\phi}_{L^2(\R^3)}^2}{2\max\big(1,8\mu(\R^3)\big)},\qquad \forall \phi\in H^1(\R^3,\C^2).
\label{eq:lower_bound_cV}
\end{equation}
Denoting $M=2\max\big(1,8\mu(\R^3)\big)\geq2$, we have
$$M\int_{\R^3}\frac{|\sigma\cdot\nabla\phi(x)|^2}{1+V_\mu(x)}\,dx\geq \pscal{\phi,|p|\phi}-\norm{\phi}_{L^2(\R^3)}^2\geq \pscal{\phi,|p|\phi}-(M-1)\norm{\phi}_{L^2(\R^3)}^2$$
so that 
\begin{equation}
 \int_{\R^3}\frac{|\sigma\cdot\nabla\phi(x)|^2}{1+V_\mu(x)}\,dx+\norm{\phi}_{L^2(\R^3)}^2\geq \frac{\pscal{\phi,|p|\phi}+\norm{\phi}_{L^2(\R^3)}^2}{M}
 \label{eq:estim_quadratic_form}
\end{equation}
which is the left side of~\eqref{eq:in_H_1/2}. The exact constant in this inequality is not important, but it is crucial that it only depends on $\mu$ through its mass $\mu(\R^3)$. 

The quadratic form on the left side of~\eqref{eq:estim_quadratic_form} defined on $C^\ii_c(\R^d,\C^2)$ is closable in the Hilbert space $L^2(\R^3,\C^2)$. This follows from the recent abstract result~\cite[Lem.~9]{SchSolTok-20}, which settles a delicate issue neglected in several previous papers on the subject. An alternative argument in our case is to use that the associated operator
$$-\sigma\cdot\nabla\frac{1}{1+V_\mu}\sigma\cdot\nabla+1$$
is well defined and symmetric in the domain $C^\ii_c(\R^d,\C^2)$, by Lemma~\ref{lem:V_mu} below. Then one can use~\cite[Thm.~X.23]{ReeSim2}. 

Finally, the domain of the closure is automatically a subspace of $H^{\frac12}(\R^3,\C^2)$ by~\eqref{eq:estim_quadratic_form}. Due to the upper bound in~\eqref{eq:in_H_1/2}, one can also define the quadratic form on $H^1(\R^3,\C^2)$ and then close it, without changing the result.

\subsubsection*{Step 2. $\cV_\mu$ coincides with the maximal space}
Next we prove the formula~\eqref{eq:cV_mu_maximal} which states that $\cV_\mu$ coincides with the maximal space on which one can give a meaning to the associated norm~\eqref{eq:norm_V}. We start by proving the following lemma.

\begin{lemma}[Regularity of $(1+V_\mu)^{\alpha}$]\label{lem:V_mu}
Let $\mu$ be a non-negative Radon measure over $\R^3$ and $V_\mu:=\mu\ast|\cdot|^{-1}$. Then $\nabla (1+V_\mu)^{\alpha}\in L^2(\R^3)$ for all $\alpha<1/2$ and we have 
\begin{equation}
 \int_{\R^3}\left|\nabla (1+V_\mu)^{\alpha}\right|^2\leq C_\alpha \,\mu(\R^3)
\end{equation}
for a constant $C_\alpha$ depending only on $\alpha$. When $\alpha=0$ we have the same estimate with $(1+V_\mu)^{\alpha}$ replaced by $\log(1+V_\mu)$. 
\end{lemma}

\begin{proof}
We write the proof for a non-negative $\mu\in C^\ii_c(\R^3)$. The general result follows from an approximation argument. Let $\Omega_0:=\{V_\mu<1\}$ and $\Omega_i:=\{2^{i-1}\leq V_\mu< 2^i\}$. Then we have (with $(1+V_\mu)^{\alpha}$ replaced by $\log(1+V_\mu)$ when $\alpha=0$)
\begin{align*}
\int_{\R^3}\left|\nabla (1+V_\mu)^{\alpha}\right|^2&=\alpha^2\int_{\R^3}\frac{|\nabla V_\mu|^2}{(1+V_\mu)^{2-2\alpha}}\\
&=\alpha^2\sum_{i=0}^\ii\int_{\Omega_i}\frac{|\nabla V_\mu|^2}{(1+V_\mu)^{2-2\alpha}}\\
&\leq\alpha^2\int_{\Omega_0}|\nabla V_\mu|^2+\alpha^2\sum_{i=1}^\ii\frac1{(1+2^{i-1})^{2-2\alpha}}\int_{\Omega_i}|\nabla V_\mu|^2.
\end{align*}
Since $-\Delta V_\mu=4\pi\mu$ we have for all $i\geq0$
$$\int_{\Omega_i}|\nabla V_\mu|^2=\int_{\R^3}\nabla V_\mu\cdot\nabla v_i=4\pi\int_{\R^3}v_i\;{\rm d}\mu\leq 4\pi 2^i\;\mu(\R^3)$$
where 
$$v_i:=\begin{cases}
\1(V_\mu\geq 1)+V_\mu\1(V_\mu< 1)&\text{for $i=0$,}\\
2^{i-1}\1(V_\mu\leq 2^{i-1})+2^{i}\1(V_\mu\geq 2^{i})+V_\mu\1(2^{i-1}< V_\mu< 2^i)&\text{for $i\geq1$.}
      \end{cases}
$$
We obtain
$$\int_{\R^3}\left|\nabla (1+V_\mu)^{\alpha}\right|^2\leq 4\pi\alpha^2\left(1+\sum_{i=1}^\ii\frac{2^i}{(1+2^{i-1})^{2-2\alpha}}\right)\;\mu(\R^3),$$
where the series is finite since $\alpha<1/2$.
\end{proof}

The lemma says that for $\phi\in L^2(\R^3)$, $(1+V)^{-1/2}\sigma\cdot\nabla\phi$ makes sense as a distribution. It is then equivalent to require the existence of $g\in L^2$ such that $\sigma\cdot\nabla\phi=(1+V_\mu)^{1/2}g$ or to ask that the distribution $(1+V_\mu)^{-1/2}\sigma\cdot\nabla\phi$ belongs to $L^2$. In the following we freely use any of the two formulations. 

Next we turn to the proof that any function $\phi$ such that $g:=(1+V_\mu)^{-1/2}\sigma\cdot\nabla\phi\in L^2(\R^3,\C^2)$ can be approximated by a sequence $\phi_n$ in $C^\ii_c(\R^3,\C^2)$ for the norm $\|\cdot\|_{\cV_\mu}$, that is, such that $\phi_n\to\phi$ in $L^2(\R^3,\C^2)$ and $(1+V_\mu)^{-1/2}\sigma\cdot\nabla \phi_n\to g$ in $L^2(\R^3,\C^2)$. 

First we truncate the sequence in space. We define $\phi_n(x):=\phi(x)\chi(x/n)$ where $\chi\in C^\ii_c$ is such that $\chi(0)=1$. We have of course $\phi_n\to\phi$ in $L^2(\R^3,\C^2)$. In addition, we have in the sense of distributions, 
\begin{align*}
\sigma\cdot\nabla \phi_n&=\chi(\cdot/n)\sigma\cdot\nabla \phi+\phi\frac{(\sigma\cdot\nabla\chi)(\cdot/n)}{n}\\
&=(1+V_\mu)^{\frac12}\left(\chi(\cdot/n)g+\phi\frac{(\sigma\cdot\nabla\chi)(\cdot/n)}{n(1+V_\mu)^{\frac12}}\right) 
\end{align*}
where the function in parenthesis has a compact support and converges to $g$ in $L^2(\R^3)$. This proves that functions of compact support are dense. In the following we assume, without loss of generality, that $\phi$ and $g$ both have a compact support. 

Next we approximate $\phi$ by a sequence in $H^1$ by arguing as in~\cite{EstLewSer-19}. Let $u\in \dot{H}^1(\R^3)$ such that $\phi=-i\sigma\cdot\nabla u$. Then we have in the sense of distributions
$$-i\sigma\cdot\nabla \phi=-\Delta u=(1+V_\mu)^{\frac12}g.$$
We have
$$u=\frac1{4\pi} \Big((1+V_\mu)^{\frac12}g\Big)\ast\frac1{|x|},\qquad  \phi=i\frac1{4\pi} \Big((1+V_\mu)^{\frac12}g\Big)\ast\sigma\cdot \frac{x}{|x|^3}.$$
Next we define
$$u_\eps=\frac1{4\pi} \Big((1+V_\mu)^{\frac12}\1(V_\mu\leq\eps^{-1})g\Big)\ast\frac1{|x|},\qquad  \phi_\eps=-i\sigma\cdot\nabla u_\eps$$
which satisfies
$$-i\sigma\cdot\nabla \phi_\eps=-\Delta u_\eps=(1+V_\mu)^{\frac12}g_\eps=i\frac1{4\pi} \Big((1+V_\mu)^{\frac12}g_\eps\Big)\ast\sigma\cdot \frac{x}{|x|^3}$$
where $g_\eps=g\1(V_\mu\leq\eps^{-1})$.
Since $(1+V_\mu)^{\frac12}\1(V_\mu\leq\eps^{-1})g\in L^2(\R^3)$ and $g$ has compact support, we have $(1+V_\mu)^{\frac12}\1(V_\mu\leq\eps^{-1})g\in L^{6/5}(\R^3)$. From the Hardy-Littlewood-Sobolev inequality, this shows that $\phi_\eps\in L^2(\R^3)$. From the definition we also have $\sigma\cdot\nabla \phi_\eps\in L^2(\R^3)$, hence $\nabla\phi_\eps\in L^2(\R^3)$ and $\phi_\eps\in H^1(\R^3)$. From the dominated convergence theorem, we have $g_\eps\to g$ in $L^2(\R^3)$ and we now have to show that $\phi_\eps\to\phi$ in $L^2$ as well. We have, again by the Hardy-Littlewood-Sobolev inequality,
\begin{align*}
\int_{\R^3}|\phi-\phi_\eps|^2&\leq C\norm{(1+V_\mu)^{\frac12}\1(V_\mu\geq\eps^{-1})g}_{L^{6/5}(\R^3)}^2\\
&\leq  C\norm{g}_{L^2(B)}^2\left(\int_B(1+V_\mu)^{\frac32}\1(V_\mu\geq\eps^{-1})\right)^{\frac23}
\end{align*}
where ${\rm supp}(g)\subset B$. The right side tends to zero when $\eps\to0$ and this shows that $\phi_\eps\to\phi$ for the norm $\|\cdot\|_{\cV_\mu}$. The density of $C^\ii_c$ is then proved using the fact that $\|\cdot\|_{\cV_\mu}$ is dominated by the $H^1$ norm and this concludes the proof of~\eqref{eq:cV_mu_maximal}, hence of Theorem~\ref{thm:def_cV_mu}.\qed

\section{Proof of Theorems~\ref{thm:self-adjoint} and~\ref{thm:min-max}}\label{sec:proofs_self_adjoint}

We consider the quadratic form
\begin{equation}
 q_\lambda(\phi):=\int_{\R^3}\frac{|\sigma\cdot\nabla\phi|^2}{1+\lambda+V_\mu}\,dx+\int_{\R^3}(1-\lambda-V_\mu)|\phi|^2
 \label{eq:def_q_lambda_bis}
\end{equation}
defined (first) on $H^1(\R^3)$ and show that it is coercive for the norm of $\cV_\mu$, after adding $C\norm{\phi}_{L^2(\R^3)}^2$ for an appropriate constant $C$. For this the estimate~\eqref{eq:in_H_1/2} on the first term is not enough because we also have to control the negative Coulomb part involving $V_\mu$. Our proof will be based on Theorem~\ref{thm:distinguished}, where we have shown that 
\begin{equation}
\norm{\sqrt{V_\mu}\frac1{D_0+is}\sqrt{V_\mu}}<1
\label{eq:estimate_i_s}
\end{equation}
for $s$ large enough.
We will explain here how this can be used to get some information on the quadratic form $q_\lambda$ in~\eqref{eq:def_q_lambda}. More precisely, we will use a similar argument as in the previous section and show that 
\begin{equation}
\max\Spec\left(\sqrt{V_\mu}\frac1{D_0+\frac{C-\eps|p|}2 (\beta+1)}\sqrt{V_\mu}\right)<1
\label{eq:estimate_C_beta_eps}
\end{equation}
for $C$ large enough and $\eps$ small enough. Recall that $C(\beta+1)/2$ has the effect of only translating the upper component by $C$, that is, to push the upper part of the essential spectrum. Like in the previous section, the estimate~\eqref{eq:estimate_C_beta_eps} implies that 
\begin{equation}
q_0(\phi)\geq\eps\norm{|p|^{\frac12}\phi}_{L^2(\R^3)}^2-C\norm{\phi}_{L^2(\R^3)}^2
\label{eq:bound_q_0}
\end{equation}
for all $\phi\in H^1(\R^3)$. This is the relation between the method of Nenciu \emph{et al}~\cite{Nenciu-76} based on estimates for the operator $K_\lambda$, and the method initiated by Esteban-Loss~\cite{EstLos-07,EstLos-08} based on the quadratic form $q_\lambda$. To our knowledge, this is the first time that such a link is established.

\subsubsection*{Step 1. Proof of~\eqref{eq:estimate_C_beta_eps}}

The proof of~\eqref{eq:estimate_C_beta_eps} is based on the following lemma.

\begin{lemma}[Relating resolvents]\label{lem:useful_resolvents}
For every $0\leq\eps\leq 1$ and $C\geq0$, we have the operator bound
\begin{multline}
 \frac1{D_0+\frac{C-\eps|p|}2 (\beta+1)}\\\leq \frac12 \left(\frac1{\alpha\cdot p+\beta+i\sqrt{C}}+\frac1{\alpha\cdot p+\beta-i\sqrt{C}}\right)+\frac{8\eps(1+C)}{|p|}.
\end{multline}
\end{lemma}

\begin{proof}
We start with the case $\eps=0$. Note that we have
$$D_0+\frac{C}2(\beta+1)=\alpha\cdot p+\left(1+\frac{C}2\right)\beta+\frac{C}2$$
whose spectrum is given by the values of the functions
$$\pm\sqrt{|p|^2+\left(1+\frac{C}2\right)^2}+\frac{C}{2}.$$
The upper function is clearly bounded from below by $1+C$ whereas the lower function is bounded above by $-1$. The gap is $(-1,1+C)$. For large $p$, the new operator behaves like $D_0$. This allows us to compute the resolvent, which we express in the form
\begin{align}
\frac{1}{D_0+\frac{C}2(\beta+1)} &=\frac{\alpha\cdot p+\left(1+\frac{C}2\right)\beta-\frac{C}2}{|p|^2+\left(1+\frac{C}2\right)^2-\frac{C^2}{4}}\nn\\
&=\frac{\alpha\cdot p}{|p|^2+1+C}+\frac{\beta}{|p|^2+1+C}+\frac{C}2\frac{\beta-1}{|p|^2+1+C}.\label{eq:useful_computation}
\end{align}
Inserting 
$$\frac{\alpha\cdot p}{|p|^2+C+1}=\frac12 \left(\frac1{\alpha\cdot p+\beta+i\sqrt{C}}+\frac1{\alpha\cdot p+\beta-i\sqrt{C}}\right)-\frac{\beta}{|p|^2+1+C}$$
we obtain the relation
\begin{multline*}
\frac{1}{D_0+\frac{C}2(\beta+1)} \\=\frac12 \left(\frac1{\alpha\cdot p+\beta+i\sqrt{C}}+\frac1{\alpha\cdot p+\beta-i\sqrt{C}}\right)
+\frac{C}2\frac{\beta-1}{|p|^2+1+C}.\label{eq:useful_computation2}
\end{multline*}
Since $\beta\leq1$ the last term is non-positive hence we find the simple inequality
\begin{equation}
\frac{1}{D_0+\frac{C}2(\beta+1)} \leq \frac12 \left(\frac1{\alpha\cdot p+\beta+i\sqrt{C}}+\frac1{\alpha\cdot p+\beta-i\sqrt{C}}\right).
\label{eq:useful_computation3}
\end{equation}

Next we consider the case $0<\eps\leq1$ and note first that 
$$D_0+\frac{C-\eps|p|}2(\beta+1)=\alpha\cdot p+\left(1+\frac{C-\eps|p|}2\right)\beta+\frac{C-\eps|p|}2$$
has the spectrum given by the values of the functions
\begin{multline*}
\pm\sqrt{|p|^2+\left(1+\frac{C-\eps|p|}2\right)^2}+\frac{C-\eps|p|}{2}\\=\pm\sqrt{|p|^2+1+C-\eps|p|+\left(\frac{C-\eps|p|}2\right)^2}+\frac{C-\eps|p|}{2}. 
\end{multline*}
Noticing that 
$$|p|^2+1+C-\eps|p|\geq |p|^2+1+C-|p|\geq \frac{|p|^2}{2}+\frac12+C>0.$$
we see that the two eigenvalues do not approach the origin, hence the operator is invertible. We can next estimate the difference by 
\begin{align}
&\left|\frac1{D_0+\frac{C-\eps|p|}2 (\beta+1)}-\frac1{D_0+\frac{C}2 (\beta+1)}\right|\nn\\
&\qquad =\frac{\eps|p|}2\left|\frac1{D_0+\frac{C-\eps|p|}2 (\beta+1)}(\beta+1)\frac1{D_0+\frac{C}2 (\beta+1)}\right|\nn\\
&\qquad \leq \frac{\eps|p|}{\left|D_0+\frac{C-\eps|p|}2 (\beta+1)\right|\;\left|D_0+\frac{C}2 (\beta+1)\right|}.\label{eq:estim_difference_resolvents}
\end{align}
From~\eqref{eq:useful_computation} we have
$$\frac1{\left|D_0+\frac{C}2 (\beta+1)\right|}\leq \frac{|p|+1+C}{|p|^2+1+C}\leq \frac{1+\frac{\sqrt{1+C}}2}{|p|}\leq \frac{3\sqrt{1+C}}{2|p|}.$$
On the other hand, using~\eqref{eq:useful_computation} with $C$ replaced by $C-\eps|p|$ we obtain
\begin{align*}
\frac1{\left|D_0+\frac{C-\eps|p|}2 (\beta+1)\right|}&\leq\frac{|p|+1+C+\eps|p|}{|p|^2+1+C-\eps|p|}\\
&\leq\frac{2|p|+1+C}{\frac{|p|^2}{2}+\frac12+C}\leq\frac{4+\sqrt{1+C}}{|p|}\leq\frac{5\sqrt{1+C}}{|p|}.
\end{align*}
Inserting this bound in~\eqref{eq:estim_difference_resolvents} gives the claimed inequality. The constants are not at all optimal and they are only displayed for concreteness. 
\end{proof}

Using Lemma~\ref{lem:useful_resolvents} and Kato's inequality, we obtain 
\begin{multline*}
\max\;\Spec\left(\sqrt{V_\mu}\frac{1}{D_0+\frac{C-\eps|p|}2(\beta+1)}\sqrt{V_\mu}\right) \leq 4\pi\eps(1+C)\mu(\R^3)\\+\frac12 \left(\norm{\sqrt{V_\mu}\frac1{\alpha\cdot p+\beta+i\sqrt{C}}\sqrt{V_\mu}}+\norm{\sqrt{V_\mu}\frac1{\alpha\cdot p+\beta-i\sqrt{C}}\sqrt{V_\mu}}\right).
\end{multline*}
We have shown in the proof of Theorem~\ref{thm:distinguished} that the two operator norms are less than 1 for $C$ large enough. Taking $\eps$ small enough then concludes our proof of~\eqref{eq:estimate_C_beta_eps}. 

\subsubsection*{Step 2. Proof of~\eqref{eq:bound_q_0} and equivalence of quadratic forms}
Let us truncate $V_\mu$ into $W_n=V_\mu\1(V_\mu\leq n)$ and notice that, by Step 1,
\begin{multline}
\max\Spec\left(\sqrt{W_n}\frac1{D_0+\frac{C-\eps|p|}2(\beta+1)}\sqrt{W_n}\right)\\
\leq \max\Spec\left(\sqrt{V_\mu}\frac1{D_0+\frac{C-\eps|p|}2(\beta+1)}\sqrt{V_\mu}\right)<1 
\end{multline}
since $W_n\leq V_\mu$ pointwise, for $C$ large enough and $\eps$ small enough. For the bounded potential $W_n$, the min-max formula and the Birman-Schwinger principle are well known. The previous condition implies that 
$$D_0-t W_n+\frac{C-\eps|p|}2(\beta+1)$$
has no eigenvalue in $(-1,0)$ for every $t\in[0,1]$. 
From the min-max principle and a continuation argument in $t$ from~\cite{DolEstSer-00}, this is equivalent to saying that 
$$q_{0,W_n}(\phi)\geq-C\norm{\phi}_{L^2}^2+\eps\norm{|p|^{\frac12}\phi}_{L^2}^2$$
for all $\phi\in H^1(\R^3,\C^2)$, where of course $q_{0,W_n}$ denotes the quadratic form with $V_\mu$ replaced by $W_n$. Passing to the limit $n\to\ii$ we obtain ~\eqref{eq:bound_q_0}. This is not yet enough to show that $q_0$ is equivalent to the square of the norm of $\cV_\mu$. Now using Kato's inequality, for $\eta<1$ we write 
\begin{align*}
&q_0(\phi)+C\norm{\phi}^2_{L^2}\\
&\qquad \geq \eta\int_{\R^3}\frac{|\sigma\cdot\nabla\phi|^2}{1+\mu\ast|x|^{-1}}\,dx -\eta\int_{\R^3}V_\mu|\phi|^2 +(1-\eta)\eps\norm{|p|^{\frac12}\phi}^2\\
&\qquad \geq \eta\int_{\R^3}\frac{|\sigma\cdot\nabla\phi|^2}{1+\mu\ast|x|^{-1}}\,dx  +\left((1-\eta)\eps-\frac\pi2\eta\mu(\R^3)\right)\norm{|p|^{\frac12}\phi}^2.
\end{align*}
After taking 
$$\eta<\frac{\eps}{\eps+\frac\pi2\mu(\R^3)}$$
we see that the quadratic form $q_0+C\|\cdot\|_{L^2(\R^3)}^2$ is equivalent to the square of the norm of $\cV_\mu$. This quadratic form is thus closable on $H^1(\R^3,\C^2)$ and its closure is also equivalent to the square of the norm of $\cV_\mu$.

\subsubsection*{Step 3. Domain and min-max principle}

Now we can apply the results of~\cite{EstLos-07,EstLos-08,SchSolTok-20} to the operator $D_0+C(\beta+1)/2-V_\mu$ and we obtain a unique self-adjoint extension, distinguished from the property that its domain is included in $\cV_\mu\times L^2(\R^3,\C^2)$. This relies on the fact that the multiplication operator $-1+V_\mu$ is essentially self-adjoint on $C^\ii_c(\R^3,\C^2)$~\cite{SchSolTok-20}. The domain is as described in the statement of the theorem. Of course, the same holds for $D_0-V_\mu$ since $C(\beta+1)/2$ is bounded.

Next we prove that the domain is included in $\cV_\mu\times\cV_\mu$. Let $\Psi=(\phi,\chi)\in\cV_\mu\times L^2(\R^3,\C^2)$ be in the domain. Then we have
$$\begin{cases}
(1-V_\mu)\phi+\sigma\cdot p\chi=f\in L^2(\R^3,\C^2),\\
-(1+V_\mu)\chi+\sigma\cdot p\phi=g\in L^2(\R^3,\C^2),
  \end{cases}$$
where the terms on the left side are interpreted as distributions. Since $\phi\in \cV_\mu\subset H^{1/2}(\R^3,\C^2)$, we have $(1+V_\mu)^{1/2}\phi\in L^2(\R^3,\C^2)$, by Kato's inequality. But the first equation can then be written in the form
$$\sigma\cdot p\chi=(V_\mu+1)^{\frac12}\left((V_\mu+1)^{\frac12}\phi+\frac{f-2\phi}{(V_\mu+1)^{\frac12}}\right)$$
where the function in parenthesis belongs to $L^2(\R^3,\C^2)$. By the characterization~\eqref{eq:cV_mu_maximal} of $\cV_\mu$, this gives immediately that $\chi\in \cV_\mu$. Therefore we have shown that 
$$\cD(D_0-V_\mu)\subset \cV_\mu\times\cV_\mu\subset H^{1/2}(\R^3,\C^4).$$
By uniqueness in $H^{1/2}$ we conclude that this extension must be the same as the one from Theorem~\ref{thm:distinguished}. The Birman-Schwinger principle was shown in~\cite{Nenciu-76,Klaus-80b}. This concludes the proof of Theorem~\ref{thm:self-adjoint}.

The validity of the min-max formulas was shown for one-center Dirac operators in~\cite{DolEstSer-00} in the domain of the distinguished extension and then in $H^{1/2}(\R^3)$ in~\cite{MorMul-15,Muller-16}. This was extended to all spaces between $C^\ii_c(\R^3,\C^4)$ and $H^{1/2}(\R^3,\C^4)$ in~\cite{EstLewSer-19}. Here we can follow the same approach. The min-max is valid in $H^{1/2}(\R^3,\C^4)$ and the density of $C^\ii_c(\R^3,\C^4)$ in $\cV_\mu$ allows to conclude that the formula must hold in all spaces in between, following the argument of~\cite{EstLewSer-19}. In particular, the numbers $\lambda_F^{(k)}$ are independent of the chosen space $F$. One notable difference is that in those works it is often assumed that $\lambda^{(1)}>-1$ but it was explained in~\cite{DolEstSer-06} how to handle the case where we only have $\lambda^{(k_0)}>-1$ for some $k_0\geq1$.

\subsubsection*{Step 4. Proof of $\lambda^{(k)}\nearrow 1$}
Let us first prove that for any positive integer $k$, $\lambda^{(k)}<1$ if $\mu\neq 0$. For every $k$ we choose a $k$-dimensional subspace of radial functions in $C^\ii_c(\R^3,\C^2)$, denoted by $W_k$. 
Let $U_R(f)=R^{-3/2}f(\cdot/R)$ be the unitary operator which dilates the function by a factor $R$. Introduce $W_{k,R}:=U_RW_k$. Then for every normalized $\phi_R=U_R\phi\in W_{k,R}$, we have
$$q_\lambda(\phi_R)=\frac1{R^2}\int_{\R^3}\frac{|\sigma\cdot\nabla \phi(x)|^2}{1+\lambda+V_\mu(Rx)}\,dx+\int_{\R^3}(1-\lambda-V_\mu(Rx))|\phi(x)|^2\,dx.$$
Decomposing $\mu=\mu\1_{B_\eta}+\mu\1_{B_\eta^c}$ with a large but fixed $\eta$, we have
$$\left|\int_{\R^3}V_\mu(Rx)|\phi(x)|^2\,dx-\int_{\R^3}V_{\mu\1_{B_\eta}}(Rx)|\phi(x)|^2\,dx\right|\leq \frac{\mu(\R^3\setminus B_\eta)\pi}{2R}\norm{\phi}_{H^{\frac12}(\R^3)}^2.$$
On the other hand, after passing to Fourier variables and using $|\widehat{\mu\1_{B_\eta}}(k)-\widehat{\mu\1_{B_\eta}}(0)|\leq C\eta|k|$, we find
$$\left|\int_{\R^3}V_{\mu\1_{B_\eta}}(Rx)|\phi(x)|^2\,dx-\frac{\mu(B_\eta)}{R}\int_{\R^3}\frac{|\phi(x)|^2}{|x|}\,dx\right|\leq\frac{C\eta}{R^2}\norm{\phi}_{H^1(\R^3)}^2.$$
In our finite-dimensional space, all the norms are equivalent, hence we obtain
$$\int_{\R^3}V_\mu(Rx)|\phi(x)|^2\,dx\geq \left(c\,\frac{\mu(B_\eta)-C\mu(\R^3\setminus B_\eta)}{R}-\frac{C\eta}{R^2}\right)\norm{\phi}_{L^2(\R^3)}^2$$
for some $c>0$ depending on $W_k$. Choosing $\eta$ large enough and $\lambda=1-\eps/R$ with $\eps>0$ small enough, we deduce that 
$q_{1-\eps/R}(\phi_R)< 0$ on $W_{k,R}$ for $R$ large enough. The min-max formula~\eqref{eq:min-max} can be reformulated in terms of the quadratic form $q_\lambda$ as in~\cite{DolEstSer-00, SchSolTok-20}
\begin{align}
 \lambda^{(k)}&=\inf\left\{\lambda\ :\ \exists W\subset \cV_\mu,\ \dim(W)=k\ :\ q_\lambda(\phi)\leq0,\ \forall\phi\in W\right\}.
 \label{eq:lambda_k_q_lambda}
\end{align}
Using the characterization on the first line, this proves that $\lambda^{(k)}\leq 1-\eps/R$, as we wanted. 

Next we prove that $\lambda^{(k)}\to1$ when $k\to\ii$. Note that $k\mapsto\lambda^{(k)}$ is non-decreasing and $<1$ by the previous step. In addition, recall that 
$$\Spec_{\rm ess}(D_0-V_\mu)=(-\ii,-1]\cup[1\cup\ii)$$
by Theorem~\ref{thm:distinguished}. From this we conclude that if we have $\lambda^{(k_0)}>-1$ for some $k_0$, then $\lambda^{(k)}$ is an eigenvalue of $D_0-V_\mu$ and it can only converge to $1$.
Let us argue by contradiction and assume that $\lambda^{(k)}=-1$ for all $k\geq1$. By the characterization in~\eqref{eq:lambda_k_q_lambda} we conclude that there exists a sequence of spaces $W_k\subset \cV_\mu$ of dimension $\dim(W_k)=k$ and $\eps_k\to0^+$ such that $q_{-1+\eps_k}$ is negative on $W_k$. By monotonicity with respect to $\lambda$, we conclude that $q_0$ is also negative on $W_k$. After extraction this provides a sequence $\phi_n\in\cV_\mu$ such that $\norm{\phi_n}_{L^2}=1$, $\phi_n\wto0$ weakly and $q_{0}(\phi_n)<0$. By~\eqref{eq:bound_q_0} we know that 
$$q_0(\phi_n)\geq \eps\norm{|p|^{\frac12}\phi_n}^2_{L^2(\R^3)}-C_\lambda\norm{\phi_n}_{L^2(\R^3)}^2$$
and this proves that the sequence $\{\phi_n\}$ is bounded in $H^{1/2}(\R^3)$. Using Kato's inequality for the negative term involving $V_\mu$ in $q_0$ and again the fact that $q_0(\phi_n)<0$, we finally obtain that the sequence $(\phi_n)$ is bounded in $\cV_\mu$. Next we pick a localization function $\chi_R(x)=\chi(x/R)$ where $\chi\in C^\ii_c(\R^3,[0,1])$, $\chi\equiv1$ on $B_1$ and $\chi\equiv0$ on $\R^3\setminus B_2$ and let $\eta_R:=\sqrt{1-\chi_R^2}$. We use the \emph{pointwise} IMS formula for the Pauli operator~\cite{BosDolEst-08} which states that 
\begin{align}
\sum_k|\sigma\cdot\nabla (J_k\phi)|^2&=\sum_k \sum_{i,j=1}^3\pscal{ \partial_i(J_k\phi),\sigma_i\sigma_j \partial_j(J_k\phi)}_{\C^2}\nn\\
&=|\sigma\cdot \nabla\phi|^2+\sum_k \sum_{i,j=1}^3\pscal{ \phi,\sigma_i\sigma_j \phi}_{\C^2}\partial_iJ_k\partial_j J_k\nn\\
&\qquad+2\Re\sum_k \sum_{i,j=1}^3\pscal{ \partial_i\phi,\sigma_i\sigma_j \phi}_{\C^2}J_k\partial_j J_k\nn\\
&=|\sigma\cdot \nabla\phi|^2+|\phi|^2\sum_k |\nabla J_k|^2
\label{eq:IMS}
\end{align}
for any real partition of unity $\sum_k J_k^2=1$. We have used that $\sigma_i\sigma_j+\sigma_j\sigma_i=0$ for $i\neq j$ in the second term of the second equality and that $2\sum_k J_k\partial_j J_k=\partial_j\sum_k J_k^2=0$ for the last term. We obtain
\begin{multline*}
q_0(\phi_n)=\int_{\R^3}\frac{|\sigma\cdot\nabla (\chi_R\phi_n)|^2}{1+V_\mu}-\int_{\R^3}V_\mu\chi_R^2|\phi_n|^2+1\\
+\int_{\R^3}\frac{|\sigma\cdot\nabla (\eta_R\phi_n)|^2}{1+V_\mu} -\int_{\R^3}V_\mu\eta_R^2|\phi_n|^2 -\int_{\R^3}\frac{|\nabla\chi_R|^2+|\nabla \eta_R|^2}{1+V_\mu}|\phi_n|^2.
\end{multline*}
For the first two terms involving $\chi_R$ we use that $q_0$ is bounded from below by~\eqref{eq:bound_q_0}, which yields
$$\int_{\R^3}\frac{|\sigma\cdot\nabla (\chi_R\phi_n)|^2}{1+V_\mu}-\int_{\R^3}V_\mu\chi_R^2|\phi_n|^2\geq -C\int_{\R^3}\chi_R^2|\phi_n|^2.$$
Hence we obtain
\begin{equation}
 q_0(\phi_n)\geq 1-C\int_{\R^3}\chi_R^2|\phi_n|^2 -\int_{\R^3}V_\mu\eta_R^2|\phi_n|^2 -\frac{C}{R^2}.
 \label{eq:estim_lower_bd_q_0_weak}
\end{equation}
We will prove that the negative terms on the right side are all small in the limit, which gives $q_0(\phi_n)\geq0$, a contradiction. We start with the second negative term. We decompose $\mu=\mu\chi_{R/4}^2+\mu\eta_{R/4}^2$ and remark that $V_{\mu\chi_{R/4}^2}\eta_{R}^2\leq C/R$ whereas 
$$\int_{\R^3}V_{\mu\eta_{R/4}^2}\eta_R^2|\phi_n|^2\leq \frac\pi2\mu(\R^3\setminus B_{R/2})\norm{\eta_R\phi_n}_{H^{1/2}}^2\leq C\mu(\R^3\setminus B_{R/2})$$
by Kato's inequality. Hence 
$$\int_{\R^3}V_\mu\eta_R^2|\phi_n|^2\leq C\left(\frac1R+\mu(\R^3\setminus B_{R/2})\right).$$
We may therefore choose $R$ large enough such that 
$$\int_{\R^3}V_\mu\eta_R^2|\phi_n|^2 +\frac{C}{R^2}\leq \frac12$$
in~\eqref{eq:estim_lower_bd_q_0_weak}. But, due to the weak convergence $\phi_n\wto0$ in $H^{1/2}$, we have for this fixed $R$ 
$$\lim_{n\to\ii}\int_{\R^3}\chi_R^2|\phi_n|^2=0$$
and this proves, as we claimed, that $q_0(\phi_n)\geq0$ for $n$ large. 

\subsubsection*{Step 5. Proof that $\lambda^{(1)}\geq0$ under condition~\eqref{eq:condition_Tix}}
We claim that for every finite positive measure $\mu$
\begin{equation}
\norm{K_0}\leq \mu(\R^3)\frac{\frac\pi2+\frac2\pi}{2}\quad\text{where}\quad K_0:=\sqrt{V_\mu}\frac{1}{D_0}\sqrt{V_\mu}. 
\label{eq:Tix_K_0}
\end{equation}
Then, the statement that $\lambda^{(1)}\geq0$ and that there are no eigenvalue in $(-1,0)$ follows from the resolvent formula~\eqref{eq:resolvent} and the Birman-Schwinger principle in Theorem~\ref{thm:self-adjoint}.
To estimate the norm of $K_0$ we introduce the free Dirac spectral projections $P_0^\pm=\1_{\R_\pm}(D_0)$ and write 
$$-\sqrt{V_\mu} \frac{P_0^-}{\sqrt{1-\Delta}}\sqrt{V_\mu}\leq \sqrt{V_\mu}\frac{1}{D_0}\sqrt{V_\mu}\leq\sqrt{V_\mu} \frac{P_0^+}{\sqrt{1-\Delta}}\sqrt{V_\mu}.$$
We obtain 
$$\norm{\sqrt{V_\mu}\frac{1}{D_0}\sqrt{V_\mu}}\leq \max_{\tau\in\{\pm\}}\norm{\sqrt{V_\mu} \frac{P_0^\tau}{\sqrt{1-\Delta}}\sqrt{V_\mu}}.$$
By charge-conjugation the two norms on the right are equal in the maximum. One can bound them by 
\begin{multline*}
\norm{\sqrt{V_\mu} \frac{P_0^+}{\sqrt{1-\Delta}}\sqrt{V_\mu}}=\norm{\frac{P_0^+}{(1-\Delta)^{\frac14}}V_\mu\frac{P_0^+}{(1-\Delta)^{\frac14}}}\\
\leq \mu(\R^3)\norm{\frac{P_0^+}{(p^2+1)^{\frac14}}\frac1{|x|}\frac{P_0^+}{(p^2+1)^{\frac14}}} =\mu(\R^3)\frac{\frac\pi2+\frac2\pi}{2}.
\end{multline*}
The inequality is by convexity in $\mu$ and translation-invariance of the operator $P^+_0(1-\Delta)^{-1/4}$. The last equality is due to Tix~\cite{Tix-98}. 
\qed

\section{Proof of Theorem~\ref{Thm-continuity}}\label{sec:prof_continuity}

We investigate here the case
$$\mu=\sum_{m=1}^M\nu_m\delta_{R_m},\qquad V_\mu(x)=\sum_{m=1}^M\frac{\nu_m}{|x-R_m|}$$
and start by providing the 

\begin{proof}[Proof of Lemma~\ref{lem:estim_q_lambda}]
We assume $M\geq2$, otherwise the result has already been proved before in~\cite{DolEstSer-00} with $d=0$.
We follow arguments from~\cite{BosDolEst-08} and introduce a smooth partition of unity $\sum_{k=1}^{M+1}J_k^2=1$ so that $J_k\equiv 0$ outside $ B_{d_k/2}(R_k) $ and $J_k \equiv 1$ in $B_{d_k/4}(R_k)$ for $k=1,...,M$, where $d_k=\min_{\ell\neq k}|R_k-R_\ell|$ is the distance from the other nuclei. Setting
$$d:=\min_{k\neq \ell}|R_k-R_\ell|$$
the smallest distance between the nuclei, we may assume that 
\begin{equation}
\label{est-gradients} 
\sum_{k=1}^{M+1} \|\nabla J_k\|_\infty^2 \le \frac{\kappa}{d^2}
\end{equation}
for a universal constant $\kappa$. 
By the IMS formula~\eqref{eq:IMS} we have 
\begin{align*}\label{decompos}
\int_{{\R^3}}\frac{|\sigma\cdot\nabla\,\phi|^2}{1+\lambda+V_\mu}\;dx&= \sum_{k=1}^{M+1} \int_{{\R^3}} \frac{ |\sigma \cdot \nabla\,(J_k \phi)|^2}{1+\lambda+V_\mu}\,dx-\int_{{\R^3}} \sum_{k=1}^{M+1} |\nabla J_k|^2\,\frac{|\phi|^2}{1+\lambda+V_\mu}\;dx\\
&\geq \sum_{k=1}^{M+1} \int_{{\R^3}} \frac{ |\sigma \cdot \nabla\,(J_k \phi)|^2}{1+\lambda+V_\mu}\,dx-\frac\kappa{d^2(1+\lambda)}\int_{{\R^3}}|\phi|^2\;dx\,.
\end{align*}
We can write, similarly as in~\cite[Sec. 1.4]{EstLewSer-19},
$$
\int_{\R^3}\frac{|\sigma\cdot \nabla (J_k\phi)|^2}{1+\lambda+V_\mu}\,dx=(1-\nu_k^2)\int_{\R^3}\frac{|\sigma\cdot \nabla (J_k\phi)|^2}{1+\lambda+V_\mu}\,dx+\nu_k^2\int_{\R^3}\frac{|\sigma\cdot \nabla (J_k\phi)|^2}{1+\lambda+V_\mu}\,dx\,,$$
and use the following Hardy-type inequality which was proved in~\cite{DolEstSer-00,DolEstLosVeg-04}:
\begin{equation}
\forall a\geq0,\qquad \int_{\R^3}\frac{|\sigma\cdot \nabla\phi(x)|^2}{a+1/|x|}\,dx+\int_{\R^3}\left(a-\frac{1}{|x|}\right)|\phi(x)|^2\,dx\geq0.
\label{eq:Hardy-type2}
\end{equation}
Using the fact that 
$$V_\mu \leq \frac{\nu_k}{|x-R_k|}+\frac{(M-1)\bar\nu}{d}$$
on the support of $J_k$ for $k=1,...,M$, with $\bar\nu:=\max(\nu_m)$ we obtain 
\begin{align}
&\nu_k^2\int_{\R^3}\frac{|\sigma\cdot \nabla (J_k\phi)|^2}{1+V_\mu+\lambda}\,dx\nonumber\\
&\qquad \geq \nu_k^2\int_{\R^3}\frac{|\sigma\cdot \nabla(J_k\phi)|^2}{1+\nu_k|x-R_k|^{-1}+\frac{(M-1)\bar\nu}{d}+\lambda}\,dx\nonumber\\
&\qquad \geq \int_{\R^3}\left(\frac{\nu_k}{|x-R_k|}-1-\lambda-\frac{(M-1)\bar\nu}{d}\right)|J_k\phi|^2\,dx\nonumber\\
&\qquad \geq  \int_{\R^3}V_\mu|J_k\phi|^2\,dx-\left(1+\lambda+\frac{2(M-1)\bar\nu}{d}\right)\int_{\R^3}|J_k\phi|^2\,dx\label{eq:estim_kinetic_potential}
\end{align}
for $k=1, \dots, M$. For $k=M+1$ we use that 
\[
 \int_{\R^3}V_\mu|J_{M+1}\phi|^2\,dx \le \frac{M\bar\nu}{d} \int_{\R^3}|J_{M+1}\phi|^2\,dx.
\] 
Since $-1<\lambda<1$ and $2(M-1)\geq M$ for $M\geq2$, we obtain
\begin{align}
 &\int_{\R^3}\frac{|\sigma\cdot\nabla\,\phi|^2}{1+\lambda+V_\mu}\;dx +\int_{\R^3} (1-\lambda+V)|\phi|^2\,dx\nn\\
 &\ \geq(1-\bar\nu^2)\sum_{k=1}^M\int_{\R^3}\frac{|\sigma\cdot \nabla (J_k\phi)|^2}{2+V_\mu}\,dx+\int_{\R^3}\frac{|\sigma\cdot \nabla (J_{M+1}\phi)|^2}{2+V_\mu}\,dx\nn\\
& \qquad-\left(2\lambda+\frac{2(M-1)\bar\nu}{d}+\frac{\kappa}{d^2(1+\lambda)}\right)\int_{\R^3}|\phi|^2\,dx\nn\\
 &\ \geq(1-\bar\nu^2)\int_{\R^3}\frac{|\sigma\cdot \nabla \phi|^2}{2+V_\mu}\,dx-\left(2\lambda+\frac{2(M-1)\bar\nu}{d}+\frac{\kappa}{d^2(1+\lambda)}\right)\int_{\R^3}|\phi|^2\,dx.
 \label{eq:estim_q_lambda}
\end{align}
\end{proof}

We are now ready to provide the

\begin{proof}[Proof of Theorem~\ref{Thm-continuity}]
To prove $(i)$, let us fix some $R_1,...,R_M$ all distinct from each other and let $d:=\min_{j\neq k}|R_j-R_k|>0$.  Let $\chi\in C^\ii_c(B_{d/4})$ be a function such that $\chi_{B_{d/8}}\equiv1$ and denote by $\chi_m:=\chi(x-R_m)$ the function centered at the $m$-th nucleus. Let finally $\eta:=1-\sum_{m=1}^M\chi_m$ be the function which localizes outside of the $M$ nuclei. Next we consider some new positions $R_1',...,R_M'\in\R^3$ such that $|R_m-R_m'|\leq \eps\leq d/10$ and define the following deformation of space
$$Tx=\sum_{m=1}^M\chi_m(x) (x+R_m'-R_m)+\eta(x) x=x+\sum_{m=1}^M\chi_m(x) (R_m'-R_m)$$
which sends each nucleus $R_m$ onto $R_m'$ and does not move the points located at a distance $\geq d/4$ away from the nuclei. We have $|Tx-x|\leq \eps$ and 
$$|DT(x)-\1_3|\leq C\eps\1\left(\frac{d}8 \leq \delta(x)\leq \frac{d}4\right),\qquad \delta(x):=\min_m|x-R_m|$$ 
hence $T$ is a $C^\ii$--diffeomorphism for $\eps$ small enough. For a function $\phi\in C^\ii_c(\R^2,\C^2)$ we define $\phi_T(x):=\phi(T^{-1}x)$. Denoting by $q_\lambda$ and $q'_\lambda$ the quadratic forms corresponding respectively to the nuclear positions $R_m$ and $R_m'$ we find after a change of variable in the integrals
\begin{align*}
q_\lambda(\phi_T)\leq q'_\lambda(\phi)+C\eps \int_{\frac{d}8 \leq \delta\leq \frac{d}4}|\nabla \phi|^2+|\phi|^2\leq q'_{\lambda-C\eps}(\phi)
\end{align*}
for $\eps$ small enough and $\lambda$ far enough from $-1$. 
By~\eqref{eq:lambda_k_q_lambda} this proves that 
$$\lambda_1\left(D_0-\sum_{m=1}^M\frac{\nu_m}{|x-R^\varepsilon_m|}\right)\le \lambda_1\left(D_0-\sum_{m=1}^M\frac{\nu_m}{|x-R'_m|}\right)+C\,\varepsilon$$
for $\eps$ small enough. 
We get the reverse inequality by using the inverse transformation. 

To prove $(ii)$, we assume for simplicity of notation that $\max\nu_m=\nu_M$. We then use the pointwise bound
$$\sum_{m=1}^M\frac{\nu_m}{|x-R_m|}\geq \frac{\nu_M}{|x-R_M|}$$
and the monotonicity of $\lambda^{(1)}$ with respect to the potential to deduce that 
$$ \lambda_1\left(D_0-\sum_{m=1}^M\frac{\nu_m}{|x-R_m|}\right)\le \lambda_1\left(D_0-\frac{\nu_M}{|x-R_M|}\right) = \sqrt{1-\nu_M^2}\,.$$
The reverse inequality in the limit $|R_j-R_k|\to\ii$ is proved by localizing exactly as in \cite[Cor.~4.7]{BosDolEst-08}.

Finally, we discuss the proof of $(iii)$. After a translation we may assume that $R_1=0$ and that $R_m\to0$ for all $m=2,...,M$. We denote by $q_{R,\lambda}$ and $q_{0,\lambda}$ the quadratic forms associated to the potentials $V_R=\sum_{m=1}^M\nu_m|x-R_m|^{-1}$ and $V_0=\sum_{m=1}^M\nu_m|x|^{-1}$, respectively. We denote by $\lambda_1(R)$ and $\lambda_1(0)$ the corresponding min-max values. From Theorem~\ref{thm:min-max} it is known that $\lambda_1(R)\geq0$ for all $R$ whenever
$$\sum_{m=1}^M\nu_m\leq\frac2{\pi/2+2/\pi},$$
as was assumed in the statement. By a continuation principle as in~\cite{DolEstSer-00} we know that the first min-max value $\lambda^{(1)}$ is non-negative and coincides with the first eigenvalue. 
Moreover, we recall from~\eqref{eq:Tix_K_0} that 
$$\|K_{R,0}\|\leq \frac{\pi/2+2/\pi}2\sum_{m=1}^M\nu_m<1,$$
where 
$$K_{R,0}:=\sqrt{V_R}\frac{1}{D_0}\sqrt{V_R}\to \sqrt{V_0}\frac1{D_0}\sqrt{V_0}:=K_{0,0}$$
strongly. Then we have $(1-K_{R,0})^{-1}\to (1-K_{0,0})^{-1}$ strongly. Since $(D_0)^{-1}\sqrt{V_R}$ converges in norm, this proves using~\eqref{eq:resolvent} that $(D_0-V_R)^{-1}\to (D_0-V_0)^{-1}$ in norm, hence that the eigenvalues converge. Note that the other parts of the statement also follow from this argument, when $\sum_{m=1}^M\nu_m< 2(\pi/2+2/\pi)^{-1}$.
\end{proof}


\end{document}